\documentclass{amsart}
\usepackage{amssymb}
\numberwithin{equation}{section}

\newcounter{proofitem}[section]

\usepackage{color}

\newtheorem{thm}[subsubsection]{Theorem}
\newtheorem{thmm}[subsection]{Theorem}
\newtheorem{cor}[subsubsection]{Corollary}
\newtheorem{lem}[subsubsection]{Lemma}
\newtheorem{lemdef}[subsubsection]{Lemma-Definition}
\newtheorem{prop}[subsubsection]{Proposition}

\newtheorem{case}[proofitem]{Case}

\newtheorem{lemm}[subsection]{Lemma}

\theoremstyle{definition}
\newtheorem{defn}[subsubsection]{Definition}
\newtheorem{defnn}[subsection]{Definition}

\newtheorem{rem}[subsubsection]{Remark}
\newtheorem{remm}[subsection]{Remark}

\newcommand{\Ko}{K^\circ}
\newcommand{\Koo}{K^{\circ\circ}}
\newcommand{\Fo}{F^\circ}
\newcommand{\Foo}{F^{\circ\circ}}
\newcommand{\Lo}{L^\circ}
\newcommand{\Loo}{L^{\circ\circ}}

\newcommand{\Kalg}{K_{alg}}
\newcommand{\Kalgo}{K_{alg}^\circ}
\newcommand{\Kalgoo}{K_{alg}^{\circ\circ}}
\newcommand{\Falg}{F_{alg}}

\newcommand{\cE}{\mathcal E}

\newcommand{\ox}{\overline x}

\renewcommand{\epsilon}{\varepsilon}
\renewcommand{\phi}{\varphi}
\renewcommand{\emptyset}{\varnothing}

\def\Sum{{\Sigma}}

\def\Hens{{\rm Hen}}

\let\cal\mathcal

\def\11{{\mathbf 1}}

\def\NN{{\mathbf N}}

\def\QQ{{\mathbf Q}}

\def\ZZ{{\mathbf Z}}

\def\cA{{\mathcal A}}
\def\cB{{\mathcal B}}
\def\cC{{\mathcal C}}
\def\cD{{\mathcal D}}
\def\cE{{\mathcal E}}
\def\cF{{\mathcal F}}

\def\cL{{\mathcal L}}

\def\cO{{\mathcal O}}
\def\cP{{\mathcal P}}

\def\cS{{\mathcal S}}
\def\cT{{\mathcal T}}
\def\cU{{\mathcal U}}

\def\cW{{\mathcal W}}

 \def\cU{{\cal U}}
 \def\cA{{\cal A}}
 \def\cB{{\cal B}}

 \def\bF{\mathbb F}
 \def\bQ{\mathbb Q}
 \def\bZ{\mathbb Z}
 \def\bR{\mathbb R}
 \def\bC{\mathbb C}
 \def\bN{\mathbb N}

\def\r{\rho}
\def\br{\overline\rho}
\def\1{\xi_1}
\def\3{\overline \xi_1}
\def\2{\xi_2}
\def\4{\overline \xi_2}

\def\Hens{{\rm Hen}}

\begin{document}

\author[Cluckers]{Raf Cluckers}%$^\mathrm{1}$}
\address{Universit\'e Lille 1, Laboratoire Painlev\'e, CNRS - UMR 8524, Cit\'e Scientifique, 59655
Villeneuve d'Ascq Cedex, France, and,
KU Leuven, Department of Mathematics,
Celestijnenlaan 200B, B-3001 Leu\-ven, Bel\-gium}
\email{Raf.Cluckers@math.univ-lille1.fr}
\urladdr{http://math.univ-lille1.fr/$\sim$cluckers}

\author[Lipshitz]{Leonard Lipshitz}%$^\mathrm{2}$}
\address{Department of Mathematics, Purdue University, 150 North University Street, West Lafayette IN 47907-2067, USA}
\email{lipshitz@math.purdue.edu}
\urladdr{www.math.purdue.edu/$\thicksim$lipshitz/}

\begin{abstract}
We give conclusive answers to some questions about definability in analytic languages that arose shortly after the work by Denef and van den Dries, \cite{DD}, on $p$-adic subanalytic sets,
and we continue the study of non-archimedean fields with analytic structure of \cite{LR3}, \cite{CLR1} and \cite{CL1}.
We show that the language ${\mathcal L}_K$ consisting of the language of valued fields together with all strictly convergent power series over a complete, rank one valued field $K$ can be expanded, in a definitial way, to a larger language corresponding to an analytic structure (with separated power series) from \cite{CL1}, hence inheriting all properties from loc.~cit., including geometric properties for the definable sets like certain forms of quantifier elimination.
Our expansion comes from adding specific, existentially definable functions, which are solutions of certain henselian systems of equations.
Moreover, we show that, even when $K$ is algebraically closed, one does {\it not} have quantifier elimination in $\cL_K$ itself, and hence, passing to expansions is unavoidable in general.
We pursue this study in the wider generality of extending non-separated power series rings to separated ones, and give new examples, in particular of the analytic structure over ${\mathbf Z}[[t]]$ that can be interpreted and understood now in all complete valued fields. In a separate direction, we show in rather large generality that Weierstrass preparation implies Weierstrass division.
\end{abstract}

\subjclass[2000]{Primary 32P05, 32B05, 32B20, 03C10, 28B10, 03C64,
14P15; Secondary 32C30, 03C98, 03C60, 28E99}

\keywords{Henselian valued fields, Tate algebras, strictly convergent power series,
subanalytic sets, quantifier elimination,
analytic structure, separated power series, non-archimedean geometry, Weierstrass preparation and division, Artin approximation, Weierstrass systems}

\title{Strictly convergent analytic structures}

\maketitle

\section{Introduction}

We study a new notion of strictly convergent analytic structure (consisting of strictly convergent power series), linked to the notion of separated analytic structure of \cite{CL1} (consisting of power series with two kinds of variables, cf.~Remark \ref{rem:sep}). We use this study to answer some natural questions that arose shortly after the work by Denef and van den Dries on $p$-adic subanalytic sets \cite{DD}.
In a separate direction, we show in a rather large generality that Weierstrass preparation implies Weierstrass division, which is the converse of the usual implication.

To explain these questions and their context, we first fix some notation.
Let $K$ be a complete, rank one valued field (not necessarily algebraically closed).
Let $\Ko$ be the valuation ring of $K$ with maximal ideal  $\Koo$. For $m\geq 0$, put  $A_m := \Ko\langle\xi\rangle = T_m(K)^\circ$, the ring of strictly convergent power series over $\Ko$ in the variables $\xi=(\xi_1,\ldots,\xi_m)$, where a power series $\sum_{i\in\NN^m}a_i \xi^i$ is called strictly convergent if its coefficients $a_i$ go to zero as $|i|$ goes to $\infty$, where we write $\xi^i$ for $\prod_{j=1}^m \xi_j^{i_j}$.
In any complete, rank one  valued field $F \supset K$, any power series $f=\sum_i a_i \xi^i\in A_m$ gives rise to a function
from $(\Fo)^m$ to $F$ by evaluating the series.
More generally, for $L$ any valued field (not necessarily complete, nor necessarily of rank one), a system $\sigma=(\sigma_m)_{m\geq 0}$ of ring homomorphisms
$$
\sigma_m : A_m \to \cF((L^\circ)^m, L^\circ),
$$
with $\cF((L^\circ)^m, L^\circ)$ the ring of $L^\circ$-valued functions on $(L^\circ)^m$, satisfying
 \begin{itemize}
 \item[(1)]\label{1)} $\sigma_0(\Koo) \subset \Loo$,
 \item[(2)]\label{2)} $\sigma_{m}(\xi_i)=$ the $i$-th coordinate function
on $(\Ko)^m$ for  $i=1,\dots,m$,  and
 \item[(3)]\label{3)} $\sigma_{m+1}$ extends $\sigma_{m}$  with the natural inclusions $A_m\hookrightarrow A_{m+1}$ and $(L^\circ)^m\hookrightarrow  (L^\circ)^{m+1}:\xi\mapsto (\xi,0)$ inducing $\cF((L^\circ)^m, L^\circ) \hookrightarrow \cF((L^\circ)^{m+1}, L^\circ)$.
\end{itemize}
is called an analytic $\{A_m\}_m$--structure on $L$.

Let $\cL_K$ be the
valued field language (including the field inverse extended by zero on zero and the division symbol $|$ satisfied by a pair $(x,y)$ if and only if $y/x$ lies in the valuation ring), together with a function symbol for each series in $\bigcup_mA_m$.
For any valued field $L$ that is equipped with an analytic $\{A_m\}_{m \in \bN}$--structure $(\sigma_m)_m$, the field $L$ can be turned into an $\cL_K$-model with the valued field structure for the valued field language and, for each function symbol $f$ corresponding to a series $f_0$ in $A_m$, the interpretation of $f$ is given by
$\sigma_m(f_0)$ on $(\Lo)^m$ extended by zero to a function on $L^m$.
An important part of this paper is the study of $\cL_K$-definable sets in Henselian valued fields with analytic $ \{A_m\}_{m \in \bN}$--structure,
although this is in fact just one of the examples that we give for our more general concept of strictly convergent analytic structure, introduced in Section \ref{formal}.  We have chosen in much of this paper to put the full field inverse in the language, instead of restricted division, as in some previous papers.  The two approaches are equivalent, see Remark \ref{AQE}.

In several contexts
(for example on $p$-adic fields), the definable sets are already well understood, but the general case has
remained rather mysterious for a long time.
In \cite{DD}, Denef and van den Dries obtained a natural quantifier elimination result for $\cL_{\QQ_p}$-definable subsets of $\QQ_p^n$, based on the semi-algebraic quantifier elimination result for semi-algebraic sets by A. Macintyre. Precisely, joining the language of Macintyre with the language $\cL_{\QQ_p}$ one obtains quantifier elimination for $\QQ_p$.
The existence of quantifier elimination in the language of valued fields for algebraically closed valued fields (of any characteristic) led to the question of whether a
similar transition could be made from the semi-algebraic sets to the subanalytic ones, that is,
for $\cL_K$-definable subsets of $K^n$ when the complete, rank one valued field $K$ is algebraically closed.
The problem is subtle.  An elaborate attempt to prove quantifier elimination in the language $\cL_K$ via a ``flattening theorem" ultimately failed --  see \cite{LR5}
for an account of this history and for a counterexample to that strategy.
In this paper we show, among other things, that such a direct analogue to the $p$-adic case is false: a complete, rank one, algebraically closed valued field $K$ does \emph{not} have quantifier elimination in the language $\cL_K$, see Section \ref{Cex}.

Recall that, by \cite{LR2}, there is a definitial expansion
$\cL_{K}^{\mbox{\tiny\cite{LR2}}}$
of $\cL_K$ such that,
$\Kalg$ has quantifier elimination in $\cL_K^{\mbox{\tiny\cite{LR2}}}$. The expansion $\cL_K^{\mbox{\tiny\cite{LR2}}}$ is obtained from $\cL_K$ by adding function symbols for certain functions, $\cL_K$-existentially definable  on $\Kalg$
that give the ``Weierstrass data" for regular powerseries.
In this paper we extend this positive result of \cite{LR2} in several directions.

Generally, for not necessarily algebraically closed $K$ and a Henselian valued field $L$ with analytic $\{A_m\}_m$--structure as above, we can still control the geometry of the $\cL_K$-definable sets,
in the sense that we give a definitional expansion $\cL_K'$ of $\cL_K$ that corresponds to a separated analytic structure as defined in \cite{CL1}.  In particular, if $L$ is assumed to have characteristic zero (and arbitrary residual characteristic), then in a Denef-Pas style expansion $\cL_K''$ of $\cL'_K$ (or just an expansion with $RV_N$-sorts, see below), as in \cite{CL1}, one still gets elimination of valued field quantifiers in the language $\cL_K''$.
The expansion $\cL_K'$ is obtained from $\cL_K$ by adding existentially definable functions, which are solutions of certain henselian systems of equations.  The proof that Weierstrass Preparation and Division hold is reminiscent of the proof of Artin Approximation.

The simpler nature
 of the $p$-adic situation is not unique. Indeed, the case that $K$ and $L$ are discretely valued, complete and of characteristic zero, and when $\sigma_0$ maps a uniformizer of $K^\circ$ to a uniformizer of $\Lo$
 is already well understood and is treated by
Theorem 4.2 of \cite{CLR1} and Example 4.4(1) of \cite{CL1}.
The general study of $\cL_K$-definable sets is more subtle and is captured by Theorem \ref{ex1}.

Another direction in which we extend the results from \cite{LR2} is by starting from rings other than the $A_m$ above, for example, instead of $A_m$ we can take the ring $B_m$ of strictly convergent power series over $\ZZ[[t]]$ in the variables $\xi=(\xi_1,\ldots,\xi_m)$, where a power series is called strictly convergent if its coefficients $t$-adically go to zero; an analytic $\{B_m\}$--structure $(\sigma_m)_m$ on a valued field $L$ is axiomatized as above, where (\ref{1)}) now reads as $\sigma_0(t\ZZ[[t]]) \subset \Loo$.
Also such analytic $\{B_m\}$--structures are related to the ones of \cite{CL1} by definitial expansions, see Section \ref{ex2a}.

More generally, we define a concept of \emph{strictly convergent analytic structure} in Section \ref{formal}, which is a general concept for systems of (non-separated) power series that can be extended, by adding certain power series, to a separated analytic Weierstrass system as in \cite{CL1}, and the above systems $\{A_m\}$ and $\{B_m\}$ turn out to be examples by Theorems \ref{ex1} and \ref{ex3}. The (separated) power series that one  adds to the non-separated ones are solutions to certain systems of henselian analytic equations, and correspond naturally to functions whose graph is existentially definable, see Section \ref{sec:hens}.

\begin{rem}[A note about separeted versus non-separated power series]\label{rem:sep}
In the present context of rings of power series and the analytic interpretations they give on Henselian valued fields, often there are one or two different kinds of variables, which play different roles. When two different kinds of variables occur, we speak of separated power series, usually denoting one kind of variables by $\xi_i$ and the other kind by $\rho_j$ with indices $i,j$. When only one kind of variables occurs, one may speak of non-separated power series, or, of strictly convergent power series. For a separated power series $f$ in $\xi$ and $\rho$, in an interpretation in a henselian valued field $L$, the variables $\xi$ will run over $L^\circ$, while the variables $\rho$ will run over $L^{\circ\circ}$ (namely, the $\xi_i$ over the valuation ring and the $\rho_j$ over the maximal ideal).
In this paper, the starting point will be power series in one kind of variables, $\xi$-variables, and, gradually, separated power series will appear, in $\xi$ and $\rho$. Indeed, the solutions of certain systems of equations may have power series solutions whose natural domain of convergence will be products of the valuation rings with the maximal ideals of the valued fields under consideration, see Section \ref{subs:sep}. Finally note that `strictly convergent' in this paper often means non-separated, and does not designate convergence in the classical sense.
\end{rem}

Previous attempts to understand $\cL_K$-definable sets include a quantifier elimination for algebraically closed valued fields in a (much) smaller language including only overconvergent power series by Schoutens (\cite{S1}, \cite{S2}, \cite{S3}),  with a correction by F. Martin in \cite{Martin1},\cite{Martin2} and in a much larger language by the second author. This larger language of \cite{LL1} consisted of rings of separated power series. The study of definable sets was continued in e.g. \cite{LL2}, \cite{LR1}, \cite{LR2}, \cite{Ce1} -- \cite{Ce5},
\cite{BMScan}, \cite{Ri}.
Dimension theory for sets definable in the larger (separated) language was developed in \cite{LR6} and a more complete dimension theory has been developed in Martin's thesis \cite{Martin1}.
A desire for uniformity and for analytic Ax-Kochen principles as in \cite{vdD} (to change the residue field characteristic), led to the notion of a henselian field with analytic $\cA$--structure, developed in \cite{vdD}, \cite{DHM}, \cite{LR3}, \cite{CLR1}, \cite{CL1}. In \cite{CL1} we gave a general theory of henselian valued fields with separated analytic structure, including a cell decomposition for definable sets and elimination of valued field quantifiers for such fields of characteristic zero.  Those results in particular apply to the separated analytic structures discussed in this paper. The issue of when a strictly convergent analytic structure on a henselian field could be extended to a separated analytic structure in a canonical way by adjoining only existentially definable functions was left largely unaddressed in \cite{CL1}, see Remark \ref{pWrem}.  We deal with this question in Section 3.

We begin our paper by showing that Weierstrass preparation implies Weierstrass division in a rather large generality. This is the converse of the generally known implication.

\section{Weierstrass Preparation implies Weierstrass Division}\label{WPWD}

In many contexts, it is well-known that the Weierstrass Division Theorem implies the Weierstrass Preparation Theorem. In this section, we show that the converse implication holds too, under very mild conditions.
Perhaps this observation has been made elsewhere, but we have not seen it before.

\subsection{The strictly Convergent case}
Let $A$ be a commutative ring and let $I \subset A$ be an ideal with $I\not=A$.

We consider a family $\{A_m\}$ of rings for $m\geq 0$ with $A_0 = A$ and for $1 \leq m \leq m'$
$$
A[\xi_1, \cdots , \xi_m] \subset A_m \subset  A_{m'} \subset A[[\xi_1, \cdots , \xi_{m'}]].
$$
For $f  = \sum a_\mu \xi^\mu \in A_m$, let $\widetilde f := \sum \widetilde {a}_\mu \xi^\mu \in \widetilde A_m$, where \ $\widetilde {} \ :A \to A/I $ is the residue map, and $ \widetilde A_m := \{\widetilde f : f \in A_m\} \subset \widetilde A[[\xi]]$. We assume the rings $A_m$ are closed under permutation of the variables and satisfy

\begin{itemize}
 \item[($*$)] For all $m$ and $m'$ with $0\leq m\leq m'$ and with $\xi = (\xi_1, \cdots , \xi_m)$ and $\xi'' = (\xi_{m+1}, \cdots , \xi_{m'})$, and for all $f$ in $A_{m'}$, say, $f=\sum_{\mu\in\NN^m }\overline{f}_{\mu}(\xi)(\xi'')^\mu$, the
$\overline{f}_{\mu}$ are in $A_{m}$.
\item[($**$)] $\widetilde A_m = \widetilde A[\xi].$
\end{itemize}

The notion of ``regularity" in a given degree $d\geq 0$ is the following.

\begin{defn}[Strictly Convergent Regular]\label{SCR}
Let $d\geq 0$ and $m\geq 0$ be integers.
A power series $f \in A_m$, say $f=\sum_{i\in\NN^m} a_i \xi^i$ is called
{\em regular in $\xi_m$ of degree $d$} when the power series $\widetilde f \in \widetilde A[\xi]$
 is a monic polynomial in $\xi_m$ of degree $d$ in $\widetilde A_{m-1}[\xi_m]$.
\end{defn}

\begin{prop}\label{2.2} If the family of rings $\{A_m\}$ satisfies the Weierstrass Preparation Theorem, then the family also satisfies the corresponding Weierstrass Division Theorem.
More precisely, suppose that:
\begin{itemize}
\item[(WP)]
For all $d\geq 0$, all $m\geq 0$, all $f \in A_m$  regular in $\xi_m$ of degree $d$ (Definition \ref{SCR}),
there exist uniquely determined unit $u\in A_{m}$ and $r_i\in A_{m-1}$ such that $f=u \cdot (\xi_m^d + \sum_{i=0}^{d-1} r_i \xi_m^i)$.
\end{itemize}
Then it follows that:
\begin{itemize}
\item[(WD)]
For all $d\geq 0$, all $m\geq 0$, all $f \in A_m$  regular in $\xi_m$ of degree $d$, and all $g \in A_m$,
there exist uniquely determined  $q \in A_m$ and $r_1, \cdots , r_{d-1} \in A_{m-1}$ such that
$$g = q\cdot f + r_{d-1} \xi_m^{d-1} + \cdots + r_0.$$
\end{itemize}
\end{prop}

\begin{proof} Below.
\end{proof}

Note that since the rings $A_m$ are closed under permutation of the variables, from Weierstrass Preparation for $f$ regular in $\xi_m$, Weierstrass Preparation for $f$ regular in other $\xi_i$ follows.

\subsection{The Separated Case}  Let $B$ be a commutative ring and let $I \subset B$ be an ideal with $I\not=B$.
We consider a family $\{B_{m,n}\}$ of rings for $m,n\geq 0$ with $B_{0,0} = B$ and for $m \leq m'$ and $n \leq n'$
$$
B[\xi_1, \cdots , \xi_m, \rho_1, \cdots ,\rho_n] \subset B_{m,n} \subset B_{m',n'} \subset B[[\xi_1, \cdots , \xi_{m'},  \rho_1, \cdots ,\rho_{n'}]].
$$
We define the residue map \ $\widetilde{}$ \ modulo $I$ as above. In particular, $\widetilde{B} = B/I$.
We assume that the rings $B_{m,n}$ are closed under permutation of the $\xi_i$ and under permutation of the $\rho_j$, and satisfy the following very mild conditions.
\begin{itemize}
 \item[($*$)] For all $m$ and $m'$ with $0\leq m\leq m'$ and with $\xi = (\xi_1, \cdots , \xi_m)$ and $\xi'' = (\xi_{m+1}, \cdots , \xi_{m'})$, and all $n$ and $n'$ with $0\leq n\leq n'$ and with $\rho = (\rho_1, \cdots , \rho_n)$ and $\rho'' = (\rho_{n+1}, \cdots , \rho_{n'})$
and for all $f$ in $B_{m',n'}$, say, $f=\sum_{\mu,\nu} \in\NN^{m'-m+n'-n }\overline{f}_{\mu \nu}(\xi,\rho)(\xi'')^\mu(\rho'')^\nu$, the
$\overline{f}_{\mu,\nu}$ are in $B_{m,n}$.
\item[($**$)] $\widetilde B[\xi,\rho] \subset \widetilde B_{m,n} \subset \widetilde B[\xi][[\rho]].$
\end{itemize}

 \begin{defn}[Separated Regular]\label{SepR}
 \item[(i)]
$f$ is called {\em regular in $\xi_m$ of degree $d$} when  in $\widetilde B[\xi]$,
$
\widetilde f  \mod (\rho_1, \cdots ,\rho_{n}).
$
is a monic polynomial in $\xi_m$ of degree $d$

 \item[(ii)]
$f$ is called {\em regular in $\rho_n$ of degree $d$} when, in $\widetilde B[\xi][[\rho]]$,
$$
\widetilde f \equiv \rho_n^d \mod (\rho_1, \cdots ,\rho_{n-1}, \rho_n^{d+1}).
$$
\end{defn}

\begin{prop} \label{2.4} If the family of rings $\{B_{m,n}\}$ satisfies the Weierstrass Preparation Theorems in both kinds of variables, then the family also satisfies the corresponding Weierstrass Division Theorems in both kinds of variables.
More precisely,
suppose that:
\begin{itemize}
\item[(WP)]
For all nonnegative integers $m,n,d$, all $f \in B_{m,n}$ which are regular of degree $d$ in $\xi_m$, resp.~in $\rho_n$,
one has that
there exist uniquely determined unit $u\in B_{m,n}$ and $r_i\in B_{m-1,n}$ such that $f=u \cdot (\xi_m^d + \sum_{i=0}^{d-1} r_i \xi_m^i)$, resp.~
there exist uniquely determined unit $u\in B_{m,n}$ and $r_i\in B_{m,n-1}$ such that $f=u \cdot (\rho_n^d + \sum_{i=0}^{d-1} r_i \rho_n^i)$.
\end{itemize}
Then it follows that:
\begin{itemize}
\item[(WD)]
 For all nonnegative integers $m,n,d$, all $f \in B_{m,n}$ which are regular of degree $d$ in $\xi_m$, resp.~in $\rho_n$, and for all $g \in B_{m,n}$
there are unique $q \in B_{m,n}$ and $r_1, \cdots , r_{d-1} \in B_{m-1,n}$ such that
$$g = q\cdot f + r_{d-1} \xi_m^{d-1} + \cdots + r_1 \xi_m + r_0,$$
resp.~there are unique $q \in B_{m,n}$ and $r_1, \cdots , r_{d-1} \in B_{m,n-1}$ such that
$$g = q\cdot f + r_{d-1} \rho_n^{d-1} + \cdots  + r_{1} \rho_n + r_0,$$
\end{itemize}

\end{prop}

Note that, since the rings $A_{m,n}$ are closed under permutation of the $\xi$ variables and the $\rho$ variables, from Weierstrass Preparation for $f$ regular in $\xi_m$ or $\rho_n$, Weierstrass Preparation for $f$ regular in the other $\xi_i$ or $\rho_j$ follows.

The polynomial case of $A$ a field (and $I = \{0\}$) is included in Definition \ref{SCR}.  Also included in Definition \ref{SCR}  and Proposition \ref{2.2} are  the rings  of strictly convergent and overconvergent power series.
Included in Definition \ref{SepR}  and Proposition \ref{2.4} are  the rings  of
 separated power series. Also included in this case (with $m=0$) are the classical cases of  formal power series over a field or complete valuation ring, algebraic power series over a field or a complete valuation ring, or germs of convergent power series over for example $\bR$ or $\bC$. (Note that the proof of Proposition \ref{2.2} in the case $m=0$ stays in that case -- i.e. Weierstrass Division in $B_{0,n}$ follows from Weierstrass Preparation in $B_{0,n+1}$).
In each of these examples it is well-known that Weierstrass Preparation is an easy consequence of Weierstrass Division, i.e. if the rings satisfy Weierstrass Division they also satisfy Weierstrass Preparation. The above propositions treat the converse implication.

\begin{proof}[Proof of Proposition \ref{2.4}]
Let $f(\xi,\rho) \in B_{m,n}$ be regular in $\xi_m$ of degree $s$.  By Weierstrass Preparation we may write
\begin{eqnarray*}
f (\xi,\rho) &=& U(\xi,\rho)\cdot [\xi_m^s + a_{s-1}(\xi',\rho)\xi_m^{s-1} + \cdots +a_0(\xi',\rho)]\\
&=&U(\xi,\rho)\cdot P(\xi,\rho),
\end{eqnarray*}
say, where $\xi' = (\xi_1, \cdots , \xi_{m-1})$ and the $a_i \in B_{m-1,n}$.  Let $g \in B_{m,n}$.  Then
$$
P(\xi,\rho) + \rho_{n+1}g(\xi,\rho)
$$
$\in B_{m,n+1}$ is also regular in $\xi_m$ of degree $s$.  Hence, by Weierstrass Preparation (in $B_{m,n+1}$)  we have
$$
P(\xi,\rho) + \rho_{n+1}g(\xi,\rho) = \cU(\xi,\rho, \rho_{n+1}) \cdot \cP(\xi,\rho, \rho_{n+1})
$$
where $\cP(\xi,\rho,\rho_{n+1})$ is a (Weierstrass) polynomial of degree $s$ in $\xi_m$.  Let
$$
\cP(\xi,\rho,\rho_{n+1}) = R_0(\xi,\rho) + \rho_{n+1}R_1(\xi,\rho) + \rho_{n+1}^2R_2(\xi,\rho) \cdots .
$$
Then $R_0(\xi,\rho) = P(\xi,\rho)$, $\cU(\xi,\rho,0) = 1$ and each $R_i(\xi,\rho)$ is a polynomial in $\xi_m$ of degree $\leq s$.  We can rewrite this equation as
\begin{equation*}
\begin{split}
\cU(\xi,\rho, \rho_{n+1}) \cP(\xi,\rho,\rho_{n+1}) &=\\
P(\xi,\rho) + \rho_{n+1}\Big[\frac{\partial \cU}{\partial \rho_{n+1}}
&
\Big|_{\rho_{n+1} = 0} P(\xi,\rho) + R_1(\xi,\rho)\Big] + \rho_{n+1}^2\Big[\cdots\Big] + \cdots .
\end{split}
\end{equation*}
Hence, equating coefficients of $\rho_{n+1}$, we see that
$$
g(\xi,\rho) = Q(\xi,\rho)\cdot P(\xi,\rho) + R_1(\xi,\rho),
$$
where $Q(\xi,\rho) = \frac{\partial \cU}{\partial \rho_{n+1}}\big|_{\rho_{n+1} = 0}$.  Hence we have obtained the Weierstrass data for the division of $g$ by $P$ from that corresponding to Weierstrass Preparation for $P + \rho_{n+1} g$.  Then $g(\xi,\rho) = Q(\xi,\rho)U(\xi,\rho)^{-1}\cdot U(\xi,\rho) P(\xi,\rho) + R_1(\xi,\rho)$, and since $f = UP$, this gives the Weierstrass data for the division of $g$ by $f$.

The case that $f$ is regular in $\rho_n$ is similar.
\end{proof}

\begin{proof}[Proof of Proposition \ref{2.2}]In the case of Proposition \ref{2.2} that we only have one sort of variable we can proceed as follows.  In the special case that $A$ is the valuation ring of a non-trivially valued field instead of
$P(\xi,\rho) + \rho_{n+1}g(\xi,\rho)$  we consider
$P(\xi) + \epsilon \xi_{m+1}  g(\xi)$, for some $\epsilon \in A$ with $0 < |\epsilon| < 1$.  This series is regular in the appropriate sense in $\xi_m$, and we can proceed exactly as in the proof of Proposition \ref{2.4}.  In the general case we can do some {\it Euclidean} division of $g$ by $f$ to write $g = q' \cdot  f + g'$ for $q'$ a polynomial in $\xi_m$ and $g'$ satisfying that $\widetilde{g'}$ is a polynomial in $\xi$ of degree  $< s$ .  Then we need only find the Weierstrass data of the division of $g'$ by $f$, and we can do that as in the  proof  of Proposition \ref{2.4} considering $f + \xi_{m+1}g'$ which is regular in $\xi_m$ of degree $s$. \end{proof}

\section{Strictly convergent Weierstrass systems and their properties}

In this section we show how to extend a strictly convergent pre-Weierstrass system to a separated pre-Weierstrass system in a uniformly, existentially definable way, by adjoining certain ``henselian" powerseries.  Under mild additional conditions this extension leads to a separated Weierstrass system.  In particular, that is the case for the basic examples (Theorems \ref{ex1} and \ref{ex3}).  In subsection \ref{formal} we discuss the general case.  The results generalize and extend the partial results about strictly convergent Weierstrass systems given in \cite{CL1}  (see Remark \ref{pWrem}).

\subsection{Definitions}
By a valued field we mean a field together with a valuation map to a (possibly trivial) ordered abelian group (not necessary of rank one) satisfying the usual non-archimedean properties.
 We denote the valuation ring of a valued field $F$ by  $\Fo$, the maximal ideal of $\Fo$ by $\Foo$ and the residue field $\Fo/\Foo$ by $\widetilde F$. By a henselian field we mean a valued field for which hensel's lemma holds on $\Fo$. For a henselian field $F$ we let $\Falg$ denote the algebraic closure of $F$; the valuation on $F$ has a unique extension to a valuation on $\Falg$.  Unless otherwise stated we write $\xi = (\xi_1, \dots , \xi_m)$ when $m$ is clear from the context.

Let $A$ be a commutative ring with unit, and let $I$ be an ideal of $A$ with $I\not=A$. Let $\widetilde A := A/I$
and write  \ $\widetilde{}: A  \to \widetilde A$ for the projection map. We also write \ $\widetilde {}$ \ for the residue map
$A[[\xi]] \to \widetilde A[[\xi]]$ which sends $\sum_\mu a_\mu \xi ^\mu$ to $\sum_\mu \widetilde a_\mu \xi ^\mu$,
and $(A_{m})\;\widetilde{}$ \  for the image of $A_m$ under this map when $A_m \subset A[[\xi]]$.

  \begin{defn}[Strictly Convergent Regular]
\label{regulars}  A power series $f \in A[[\xi]]$ is called
{\em regular in $\xi_m$ of degree $d$} when $\widetilde f \in \widetilde A[[\xi]]$ is  a monic polynomial in $\xi_m$ of degree $d$.
\end{defn}

 \begin{defn}[System]\label{PWSs}
A collection $\cA=\{A_{m}\}_{m\in \bN}$ of $A$-algebras $A_{m}$,
satisfying for all $m\geq 0$ that
\begin{itemize}
\item[(i)]
$ A_{0} = A,$
\item[(ii)]
$A_{m}$ is a subalgebra of $A[[\xi_1,\ldots,\xi_m]]$ which is closed under permutation of the variables,
\item[(iii)]
$A_{m}[\xi_{m+1}]\subset A_{m+1}$,
\item[(iv)]
$(A_{m})\;\widetilde{}\;$ is the polynomial ring
$\widetilde{A}[\xi_1,\cdots,\xi_m]$,
\end{itemize}
is called a
system, or an $(A,I)$-system, if we want to emphasize the dependence on $A$ and $I$.
\end{defn}

\begin{defn}[Analytic structure]
\label{AS}
Let $\cA = \{A_{m}\}$ be a system
as in Definition \ref{PWSs}, and let $K$ be a valued field. A
\emph{strictly convergent analytic $\cA$-structure on $K$} is a
collection of homomorphisms $\sigma = \{\sigma_{m}\}_{m\in\bN}$, such that,
for each $m\geq 0$, $\sigma_{m}$ is a homomorphism from $A_{m}$ to
the ring of $K^\circ$-valued functions on $(K^\circ)^m$ satisfying:
 \begin{itemize}
 \item[(1)] $I\subset\sigma_{0}^{-1} (K^{\circ\circ})$,
 \item[(2)]$\sigma_{m}(\xi_i)=$ the $i$-th coordinate function
on $(\Ko)^m$, $i=1,\dots,m$,  and
 \item[(3)] $\sigma_{m+1}$ extends $\sigma_{m}$ where we
identify in the obvious way functions on $(\Ko)^m$ with functions on
$(\Ko)^{m+1}$ that do not depend on the last coordinate.
\end{itemize}
\end{defn}

 \begin{defn}[Pre-Weierstrass System]\label{PWS}
 A system $\cA=\{A_{m}\}_{m\in \bN}$ of $A$-algebras $A_{m}$,
satisfying, for all $m \leq m'$:
\begin{itemize}
 \item[(v)]If  $f \in A_{m'}$, say $f =
\sum_{\mu}\overline{f}_{\mu}(\xi)(\xi'')^\mu$, then the
$\overline{f}_{\mu}$  are in $A_{m}$, where $\xi'' = (\xi_{m+1}, \cdots , \xi_{m'})$ and $\xi = (\xi_1,\cdots,\xi_m)$,
 \item[(vi)] (Weierstrass preparation) If $f \in A_m$ is regular in $\xi_m$ of degree $d$, then there exist
uniquely determined unit $u\in A_{m}$ and $r_i\in A_{m-1}$ such that $f=u \cdot (\xi_m^d + \sum_{i=0}^{d-1} r_i \xi_m^i)$,
\item[(viiWNP)](Weak noetherian property) If $f = \sum_\mu a_\mu \xi^\mu \in A_m$ there is a finite set $J \subset \bN^m$, and for each $\mu \in J$ a power  series $g_\mu \in A_m$ with $\widetilde g_\mu = 0$ such that
     $$
     f = \sum_{\mu \in J} a_\mu \xi^\mu (1 + g_\mu).
     $$
\item[(viii)] The ideal
$
\{a \in A : \sigma_0(a) = 0 \text{ for all analytic structures } \sigma=(\sigma_m) \}
$
equals the zero ideal.
\end{itemize}
is called a \emph{strictly convergent
 \textbf{pre}-Weierstrass System}, or an $(A,I)$-strictly convergent
 pre-Weierstrass System, if we want to emphasize the dependence on $A$ and $I$.
\end{defn}
Let us fix, until subsection \ref{ex1a}, a strictly convergent pre-Weierstrass System $\{A_m\}_m$.

Later on, a stronger form of (viiWNP) will play a role.

\begin{rem}\label{weakSNP}
\item[a)]
By the results of Section \ref{WPWD}, from condition (vi) we also have that Weierstrass Division  holds in the $A_m:$ \\
{\it If $f \in A_m$ is regular in $\xi_m$ of degree $d$, and $g \in A_m$ there are unique $q \in A_m$ and $r_1, \cdots , r_{d-1} \in A_{m-1}$ such that
$
g = q\cdot f + r_{d-1} \xi_m^{d-1} + \cdots + r_0.
$
Furthermore, if $\widetilde g = 0$ then $\widetilde q = \widetilde r_0 = \cdots = \widetilde r_{d-1} =0$.}
\item[b)] It follows from Weierstrass Preparation that if $f \in A_m$ satisfies $\widetilde f = 1$ (i.e. $f$ is regular of degree $0$) then $f$ is a unit.
\item[c)] From (viiWNP), it follows for $f \in A_m$ with
$\widetilde f = 0$ that $f \in I \cdot A_m$.
The above conditions (i) -- (vi) do not impose much structure on ``small" power series (e.g. those in $I \cdot A_m$) as they do not provide a mechanism for dividing by a biggest coefficient and bringing such a power series to the ``top" level where conditions (iv) and (vi) apply.
Condition (viiWNP) (and even more so, Definition \ref{SNPstr} below)  impose  structure on ``small" power series in the $A_m$ and provide such a mechanism.
\end{rem}

\begin{lemdef}\label{gn}
Let $\cA$ be a pre-Weierstrass system and let $K$ be a henselian field with analytic $\cA$--structure $\sigma$. Let $m\geq 0$ be an integer. We associate with $f\in A_m$ the power series
 $$
\cS^\sigma(f) := \sum_\mu \sigma(a_\mu)\xi^\mu \mbox{ in } \Ko[[\xi]].
 $$
Then the map $\cS$ is an $A$-algebra homomorphism. Moreover, if $\cS^\sigma(f) = 0$ then
$f^\sigma = 0$. Hence, the homomorphism $\sigma_m$ factors through $A_m^\sigma$, the image of $\cS$.
On $A_m^\sigma$ we can define the gauss norm as follows, with $f=\sum_\mu a_\mu \xi^\mu$ and where $J$ is as in (viiWNP)
 $$
 ||\cS^\sigma(f)|| := \max \{|\sigma(a_\mu)| : \mu \in J \}.
 $$
\end{lemdef}
\begin{proof}
That $\cS^\sigma(f) = 0$ implies
$f^\sigma = 0$ follows immediately from condition (viiWNP). The rest is clear.
\end{proof}

\begin{rem}\label{comp}
\item[a)]
If the valuation on $K$ is non-trivial, then one also has the converse implication that $f^\sigma = 0$ implies that $\cS^\sigma(f) = 0$. This also follows from condition (viiWNP).
\item[b)] Weierstrass Division for the rings $A_m$ implies that composition of tuples of power series in the $A_m$ corresponds again to power series in $A_m$, as explained in Remark 4.5.2. of \cite{CL1}.
\end{rem}

\subsection{Extension by henselian functions} \label{sec:hens}

\subsection*{}\label{subs:sep}
We sometimes wish to consider some of the variables to be of the first kind
(namely to indicate that their natural domain is the valuation ring), and some of the second kind (namely to indicate that their intended domain is the maximal ideal),
 cf.~Remark \ref{rem:sep}.
We will use $\rho  = (\rho_1, \dots , \rho_n)$ to denote variables of the second kind, and keep $\xi=(\xi_1,\ldots,\xi_m)$ for variables of the first kind. Variables of the second kind typically arise in making a power series regular, or guaranteeing that the existence of a solution to an equation, or system of equations,  is implied by Hensel's lemma.
We will sometimes write $A_{m,n}$ for $A_{m+n}$ to indicate that the last $n$ variables will be considered to be of the second kind and run over the maximal ideal; this change in notation will be convenient to gradually introduce genuinely new rings of separated power series.
This is crucial in our use of the henselian functions $h_{\Sigma}$ defined below. These $h_{\Sigma}$ are solutions of henselian systems (also defined below), and correspond to unique power series over $A$, which may be new, that is, not lie in any of the $A_m$.
For example, the power series $1+y - \rho_1 \cdot y^2$ is regular of degree $1$ in $y$ (Definition \ref{SepR}(i) or Definition \ref{regular}(i), and  the equation $1+y - \rho_1 \cdot y^2 =0$ defines $y=y(\rho)$ as a henselian function of (indeed, as a henselian power series in) $\rho_1$ since $\rho_1$ is considered to run over the maximal ideal, but the equation $ 1 + y - \xi_1 \cdot y^2$ only defines $y$ as a henselian function of $\xi_1$ when $|\xi_1| < 1$, and does not define $y$ as a power series in $\xi_1$; indeed not even as a function of $\xi_1$ when $|\xi_1| = 1$.
We will later enrich the $A_m$ by adding power series like $y(\rho)$, and get a richer system, genuinly in the two kinds of variables.

Basic henselian functions are defined in Definition \ref{hf1}.  A
more general kind of henselian witnessing functions than the basic one  will be defined using  systems of equations of  the following form.

\begin{defn}[Systems of henselian equations]\label{hf2}
 Let $\eta = (\eta_1,\cdots,\eta_M), \lambda = (\lambda_1,\cdots,\lambda_N)$. A {\it system,
 $\Sigma$, of henselian equations for $(\eta,\lambda)$ over $(\xi,\rho)$} is a system of the form
\begin{equation}%\label{hs}
\begin{split}
0=\eta_i - b_{0,i}(\xi,\rho) - b_{1,i}(\xi,\rho,\eta,\lambda) \quad i=1,\cdots,M\\
0 = \lambda_j - c_{0,j}(\xi,\rho) - c_{1,j}(\xi,\rho,\eta,\lambda) \quad j=1,\cdots,N
\end{split}
\end{equation}
where the
\begin{equation}\label{hss}
\begin{split}
\quad b_{0,i} &\in A[\xi,\rho], \qquad \quad b_{1,i} \in (I,\rho,\lambda)A_{m+M,n+N}, \qquad  \text{  and the }\\
\quad c_{0,j} &\in (\rho)A[\xi,\rho], \qquad c_{1,j} \in (I,\rho,\lambda^2)A_{m+M,n+N}.
\end{split}
\end{equation}
Here $\lambda^2$ is shorthand for $\{\lambda_i\lambda_j : i,j = 1, \cdots , N \}$. By the ideal $I_\Sigma$ of the system $\Sigma$ is meant the ideal of $A_{m+M,n+N}$ generated by the right hand sides of the equations in $\Sigma$.
When $\Sigma$ is a set of {\it polynomial} equations in $(\xi,\rho,\eta, \lambda)$ we call $\Sigma$ a {\it polynomial henselian system}.
We will often write a henselian system in the equivalent form:
\begin{equation}\label{hs'}
\begin{split}
\eta_i=  b_{0,i}(\xi,\rho) + b_{1,i}(\xi,\rho,\eta,\lambda) \quad i=1,\cdots,M\\
\lambda_j =   c_{0,j}(\xi,\rho) + c_{1,j}(\xi,\rho,\eta,\lambda) \quad j=1,\cdots,N.
\end{split}
\end{equation}
\end{defn}

\begin{defn}\label{sol}
Let $g$ be in $A_{m+n+s}$ for some integer $s$ and consider power series $h_i$ in $A_{m}[[\rho_1,\ldots,\rho_n]]$ for $i=1,\ldots,s$.
Then we call $h=(h_i)_{i=1}^s$ a power series solution of the equation $g = 0$
when for each integer $k> 0$ there exists a tuple $p_k$ of polynomials in the variables $\rho$ of degree $<k$ with coefficients in $A_m$ such that $h \equiv p_k \mod (\rho)^k$ and $g(\xi,\rho,p_k(\xi,\rho))$ vanishes modulo $(\rho)^k$, where the composition is as in Remark \ref{comp} b). We write $g(\xi,\rho,h(\xi,\rho))=0$ to denote that $h$ is a power series solution for $g$.
\end{defn}

The next lemma shows that any system $\Sigma$ of henselian equations
has a unique power series solution in $A_m[[\rho]]$.

\begin{lem}\label{lemps} Let $\Sigma$ be a system of henselian equations as in Definition \ref{hf2}.
Then $\Sigma$ has a unique power series solution, namely, there exist unique power series $h_i$ in $A_{m}[[\rho_1,\ldots,\rho_n]]$ for $i=1,\ldots,M+N$ such that $g(\xi,\rho,h(\xi,\rho))=0$ for each $g$ in $I_\Sigma$ with $h=(h_i)_i$. We denote the tuple $(h_i)_i$ by $h_\Sigma$.
\end{lem}
\begin{proof}We prove this by induction on $n$, with notation from \ref{hf2}.  For the induction step, assume $n > 0$.
 We proceed exactly as in the proof of Lemma \ref{Wprep}, below. Let $k >0$ and write $\eta_i = \sum_{\ell = 0}^k \eta_{i\ell}\rho_n^\ell$ and $\lambda_j = \sum_{\ell = 0}^k \lambda_{j\ell}\rho_n^\ell$, (where the  $\eta_{i\ell}, \lambda_{j\ell}$ are new variables). Substitute into the system of equations $\Sigma$ and do Weierstrass division by $\rho_n^k$ to obtain a henselian system that, by induction, determines the $\eta_{i\ell}$ and $\lambda_{j\ell}$ as power series in $\xi$ and $\rho_1, \cdots, \rho_{n-1}$.  See the proof of Lemma \ref{Wprep} for details.  Since $k$ is arbitrary, this determines $\eta$ and $\lambda$ as power series.

We must still
treat the case $n=0$, $m$ arbitrary. In this case we have that the $\eta_i, \lambda_j$ are defined by a system $\Sigma$ of the form
\begin{equation}\label{hs'}
\begin{split}
\eta_i &= b_{0,i} + b_{1,i}(\eta,\lambda) \quad i=1,\cdots,{m+M}\\
\lambda_j &= c_{0,j} + c_{1,j}(\eta,\lambda) \quad j=1,\cdots,N
\end{split}
\end{equation}
where the
\begin{equation}\label{hss'}
\begin{split}
\quad b_{0,i} &\in A_m, \qquad \quad b_{1,i} \in (I,\lambda)A_{{m+M},N}, \qquad  \text{  and the }\\
\quad c_{0,j} &\in IA_m, \quad \qquad c_{1,j} \in (I,\lambda^2)A_{{m+M},N}.
\end{split}
\end{equation}
If we allow the $c_{0,j} \in  A_{m+M}$ with $\widetilde c_{0,j} = 0$,  may assume that the $ c_{1,j}$ all have $0$ as constant (i.e. $\lambda$-free) term, and, after an $A_{{m+M}}$--linear change of variables among the $\lambda_j$ with Jacobian a unit in $A_{m+M}$, we may assume that the $c_{1,j}$ have no linear terms in $\lambda$.  Below we will see that the equations
\begin{equation}\label{hs''}
\lambda_j = c_{0,j} + c_{1,j}(\eta,\lambda) \quad j=1,\cdots,N
\end{equation}
with the $c_{0,j} \in A_{m+M}, \widetilde c_{0,j} = 0$ and the $c_{1,j} \in (\lambda^2)A_{{m+M},N}$,
have a unique solution with the $\lambda_j= \lambda_j(\eta) \in A_{m+M}$ and the $\widetilde  \lambda_j(\eta) =0$.  Substituting into the equations
$$
\eta_i = b_{0,i} + b_{1,i}(\eta,\lambda(\eta)) \quad i=1,\cdots,{m+M}
$$
yields a system of equations for the $\eta_i$, all regular of degree $1$, which we can solve successively by (strictly convergent) Weierstrass Preparation and Division, to get the solution with the $\eta_i \in A_m$.  Substituting into the solution for the $\lambda_j$ determines these as elements of $I\cdot A_m$.

Finally, it remains to show that the equations (\ref{hs''}) with the $c_{0,j} \in  A_{m+M}$ satisfying $\widetilde c_{0,j} = 0$ and the $c_{1,j} \in (\lambda^2)A_{{m+M},N}$ have a unique solution in $ A_{m+M}$ with $\widetilde \lambda = 0$.  In equation (\ref{hs''}) make the substitution
$$
\lambda_j = \sum _\ell c_{0,\ell}U_{j,\ell}
$$
where the $U_{j,\ell}$ are new variables.  This yields the equations
$$
\sum_\ell c_{0,\ell}U_{j,\ell} = c_{0,j} + \sum_{\alpha,\beta}c_{0,\alpha}c_{0,\beta}U_{0,\alpha}U_{0,\beta}G_{j\alpha \beta}(\eta,U)
$$
for $j = 1, \cdots, N$.
In place of the first equation (i.e.~$j=1$)
consider the system of equations
\begin{equation}
\begin{split}
c_{0,1}U_{1,1} &= c_{0,j} + c_{0,j}H_{1,1}(U)\\
c_{0,2}U_{1,2} &= c_{0,2} H_{1,2}(U)\\
\cdots  &\quad \cdots \\
c_{0,N}U_{1,N} &= c_{0,N} H_{1,N}(U),
\end{split}
\end{equation}
where in $c_{0,1}H_{1,1}$ we have collected all the terms of $\sum_{\alpha,\beta}c_{0,\alpha}c_{0,\beta}U_{0,\alpha}U_{0,\beta}G_{j\alpha \beta}(\eta,U)$ that have a factor $c_{0,1}U_{0,1}$, and in $c_{0,2}H_{1,2}$ all the remaining terms that have a factor $c_{0,2}U_{0,2}$, and so on.  Write down the analogous systems for $j = 2, \cdots , N$.  Now consider the system
\begin{equation}
\begin{split}
U_{1,1} &= 1 + H_{1,1}(U)\\
U_{1,2} &=  H_{1,2}(U)\\
\cdots  &\quad \cdots \\
U_{1,N} &=  H_{1,N}(U),
\end{split}
\end{equation}
and the analogous equations for  $j = 2, \cdots , N$.  In this system all the $H_{\alpha,\beta} \in  A_{{m+M} + N^2}$ satisfy $\widetilde H_{\alpha,\beta} = 0$, and hence the system can be solved for the $U_{\alpha,\beta}$ (as elements of $A_{m+M}$) by successive (strictly convergent) Weierstrass preparation and division.
This proves the existence of a solution $\overline \lambda(\eta) \in A_{m+M}^N$ to equations (\ref{hs''}).  For uniqueness substitute $\lambda_j = \overline \lambda_j + \mu_j, j= 1, \cdots, N$ in equations (\ref{hs''}), where the $\mu_j$ are new variables.
We must show that the resulting equations (in the $\mu_j$) which are of the form
\begin{equation}
\mu_j = \sum_\alpha d_\alpha \mu_\alpha + \sum_{\alpha, \beta} e_{\alpha,\beta}\mu_\alpha \mu_\beta
\end{equation}
where the $d_\alpha \in (\lambda)A_{{m+M},N}$ and the $e_{\alpha,\beta} \in A_{{m+M},N}$,
have only the zero solution.  By Condition (viiWNP) and Condition (viii) (this is the only place we use condition (viii)) it is sufficient to prove this in $A^\sigma_{{m+M},N}$ for each analytic structure $\sigma$.  But in $A^\sigma_{m+M}$ we have a Gauss-norm by \ref{gn} and it is easy to see that a non-zero solution to these equations is impossible -- suppose that we had a non-zero solution $\mu(\eta)$.  Look at the $j$th equation where $\mu_j(\eta)$ has biggest Gauss-norm. The Gauss-norm of the righthand side would be smaller than that of the lefthand side.
This completes the proof of the Lemma.
\end{proof}

\begin{defn}\label{compseries}
Let $h$ be the power series solution of a system $\Sigma$ of henselian equations as in Lemma \ref{lemps} and let $g$ be any power series in $A_{m+M,n+N}$. Write $g$ as $g_0+g_1$ where $g_0$ is in $A[\xi,\rho]$ and $g_1$ in $(I,\rho,\lambda)A_{m+M,n+N}$. Consider the system $\Sigma'$ for $(\eta,\eta_{M+1},\lambda)$ over $(\xi,\rho)$ consisting of $\Sigma$ and of the additional equation
$$
0 = \eta_{M+1} - g_0 - g_1.
$$
Then the $M+1$-th entry $h_{M+1}$ of the unique power series solution $h_{\Sigma'}=(h_i)_{i=1}^{M+1+N}$ of $\Sigma'$ is called the composition of $g$ with
the power series tuple $(\xi,\rho,h_\Sigma)$ and is denoted by $g(\xi,\rho,h_\Sigma(\xi,\rho))$.
\end{defn}

The proof of Lemma \ref{lemps} actually establishes the following version of the implicit function theorem in the power series rings $A_{m,n}=A_{m+n}$.

\begin{lem}\label{IFT} Let $\cA$ be a strictly convergent pre-Weierstrass system.
\item[(i)] Suppose that
$F_1(\xi,\eta), \cdots , F_N(\xi,\eta) \in A_{m+M}$ and  that $\overline y_1, \cdots , \overline y_N \in A_m$
satisfy $F_i(\xi,\overline y) \widetilde{~ } =0$ and $J(\xi, \overline y)$ is a unit in $A_m$, where
$J = |(\frac{\partial F_i}{\partial \eta_j})_{i,j=1,\cdots , N}|$.  Then there is a unique $y \in A_m^N$ such that
$F(\xi,y)= 0$ and $\widetilde y = \widetilde{ \overline y}$.
\item[(ii)]Suppose that $F_1(\xi,\rho,\eta,\lambda), \cdots , F_N(\xi,\rho,\eta,\lambda) \in A_{m+M, n+N}$ and  that polynomials $\overline y_1, \cdots , \overline y_{M+N} \in A_{m,n}$
satisfy $F_i(\xi,\rho,\overline y){ \widetilde{~ }} \equiv 0 \mod (\rho)$ and $J(\xi,\rho, \overline y) {\widetilde{ ~}}\equiv 1 \mod (\rho),$ where
$$J = \Big|\Big(\frac{\partial F_i}{\partial \eta_j, \partial \lambda_k}\Big)_{\underset {j=1,\cdots , M, k = 1 \cdots, N}{i=1,\cdots , M+N}}\Big|.$$  Then there is a unique $y \in A_m[[\rho]]^{M+N}$  such that  $\widetilde y \equiv \widetilde{ \overline y} \mod (\rho)$ and $F(\xi,y)= 0$, where the latter means that $y$ is a solution in the sense of Lemma \ref{lemps}.
In fact (by Proposition \ref{CH}, below) there is a polynomial henselian system $\Sigma'$ and a $G \in A_{m+n+M+N}^{M+N}$ such that $y(\xi,\rho) = G (\xi,\rho, h_{\Sigma'})$.
\end{lem}

\begin{defn}
\label{hf22}
By Lemma \ref{lemps} and Definition \ref{compseries}
we can define the following
subalgebras of $A[[\xi,\rho]]$ by
$$
A_{m,n}^H := \{g(\xi,\rho,h_\Sigma(\xi,\rho)) : g \in A_{m+M,n+N} ,
\mbox{ $\Sigma$ a henselian system}  \}
$$
and
$$
A_{m,n}^{H'} := \{g(\xi,\rho,h_\Sigma(\xi,\rho)) : g \in A_{m+M,n+N} , \mbox{ $\Sigma$ a polynomial henselian system} \}.
$$

Put
$\cA^H := \{A^H_{m,n}\}_{m,n}$.
\end{defn}

\begin{prop}\label{CH}  Let $\cA$ be a strictly convergent pre-Weierstrass system.
One has $A_{m,n}^H = A_{m,n}^{H'}$. Indeed,
every element $f$ of $A_{m,n}^H$ is of the form $f(\xi,\rho) = g(\xi,\rho,h_\Sigma(\xi,\rho)) $  where $g \in A_{m',n'}$ for some $m',n' \in \bN$, and $h_\Sigma$ is the solution to a {\bf polynomial} henselian system.
\end{prop}

\begin{proof}
 Let
\begin{eqnarray*}
\eta_i &=& b_{0,i}(\xi,\rho) + b_{1,i}(\xi,\rho,\eta,\lambda) \quad i=1,\cdots,M\\
\lambda_j &=& c_{0,j}(\xi,\rho) + c_{1,j}(\xi,\rho,\eta,\lambda) \quad j=1,\cdots,N
\end{eqnarray*}
be a henselian system.  We will obtain the solution as the composition of functions in $\cA$ and the solution to a {\it polynomial} henselian system.  We can find polynomials $ \overline b_{1,i}, \overline c_{1,j}$ such that all the coefficients of
$ b_{1,i} -  \overline b_{1,i},  c_{1,j} - \overline c_{1,j}$ are $ \in I \cdot A$, i.e. $(b_{1,i} -  \overline b_{1,i}) \widetilde{~~} = (c_{1,j} - \overline c_{1,j})\widetilde{~~} = 0,$ or equivalently, using condition (viiWNP),
$b_{1,i} -  \overline b_{1,i}, c_{1,j} - \overline c_{1,j} \in I\cdot A_{m+M+n+N}.$
Let $\overline y = (\overline y_1, \cdots , \overline y_M, \overline z_1, \cdots , \overline z_N)$ be the solution to the polynomial henselian system
\begin{eqnarray*}
\eta_i &=& b_{0,i}(\xi,\rho) + \overline b_{1,i}(\xi,\rho,\eta,\lambda) \quad i=1,\cdots,M\\
\lambda_j &=& c_{0,j}(\xi,\rho) + \overline c_{1,j}(\xi,\rho,\eta,\lambda) \quad j=1,\cdots,N,
\end{eqnarray*}
so the $\overline y_i, \overline z_j$ are terms of $\cL_\cA^{D,H'}$ (see Definition \ref{lang1}, below).  Make the change of variables $\eta = \overline y +  \zeta$ and $\lambda = \overline z + \zeta'$.  Proceed exactly as in the last stage of the proof of Lemma \ref{lemps}, (i.e. use Lemma \ref{IFT}) writing $\zeta_i = \sum_\alpha \epsilon_\alpha U_{i,\alpha}$ and
$\zeta'_j = \sum_\alpha \epsilon_\alpha U'_{j,\alpha}$ to obtain a system of equations for the $U_{i,\alpha}, U'_{j,\alpha}$ that can be solved by (strictly convergent) Weierstrass division.
\end{proof}

\begin{lemdef}[Henselian functions]\label{hf2bis}
Let $\sigma$ be an analytic $\cA$-structure on a Henselian valued field $K$.
For any henselian system $\Sigma$ and any $\xi\in(\Ko)^m$ and $\rho\in (\Koo)^n$, there exist unique values $\eta\in (\Ko)^M$ and $\lambda\in (\Koo)^N$ such that $f^\sigma(\xi,\rho,\eta,\lambda)=0$ for all $f\in I_\Sigma$; we denote this tuple $(\eta,\lambda)$ by   $k_\Sigma^\sigma(\xi,\rho)$.
Moreover, the graph of the function $k_\Sigma^\sigma$ on $(\Ko)^m\times (\Koo)^N$ is quantifier free definable in $\cL_\cA$ by a formula which is independent of $\sigma$.
\end{lemdef}

\begin{proof}
For a  polynomial system $\Sigma$ this is clear by the usual form of Hensel's lemma for systems of polynmial equations with Jacobian of norm $1$, see (i) of Lemma \ref{hf3} below. For a general $\Sigma$ the lemma now follows from the observation that the argument of Proposition \ref{CH} applies also in any field with analytic $\cA$--structure, and hence applies also to the $k_\Sigma^\sigma$.
\end{proof}

\begin{rem}For henselian system $\Sigma$ we have defined $h_\Sigma$ as a power series and $k_\Sigma$ as a (new) function symbol (or term) whose interpretation $k_\Sigma^\sigma$ in a field with analytic $\cA$--structure is defined in Lemma-Definition \ref{hf2bis}.  Later we will extend strictly convergent analytic structures $\sigma$ to separated analytic structures by defining $h_\Sigma^\sigma = k_\Sigma^\sigma$.  This will require an additional assumption (``goodness") on $\cA$ (Definition \ref{good}).
\end{rem}

The following henselian witnessing functions are much more basic than the ones associated to a system of equations in \ref{hf2bis}. Nevertheless, as we will see by (ii) of Lemma \ref{hf3}, they can in fact replace the $h_\Sigma$.

\begin{defn}[Basic henselian functions] \label{hf1}  For $K$ a henselian field $h_n : K^{n+1} \to K$ is the function that associates to $(a_0, \cdots , a_{n}, b) \in \Ko$ the unique zero, $c$, of the polynomial $p(x) := a_n x^n + a_{n-1}x^{n-1} + \cdots + a_0$ that satisfies $|c-b| < 1$, if $|p(b)| < 1$ and $|p'(b)| =1$. (Let $h_n$ output $0$ in all other cases).  The graphs of these functions are obviously quantifier-free definable in $\cL$.
\end{defn}

\begin{defn}[Languages]\label{lang1}
For $L$ a language containing the language of valued fields $\cL:=(+,\cdot,^{-1},0,1,\mid)$, let $L^D$ be $L$ with the field inverse replaced by the binary function symbol $D$ for restricted division (or two such symbols in the separated case). Further, let
$L_A$ be $L$ with constants for the elements of $A$ adjoined,
let $L_\cA$ be $L$ with function symbols for the elements of $\bigcup_m A_m$ adjoined, and let $L^{h}$ be $L$ with function symbols for the basic henselian functions $h_{n}$ adjoined.
Recall that $\mid$ is a binary symbol on a valued ring (or valued field) such that $x\mid y$ holds if and only if $x$ is nonzero and the valuation of $x$ is less or equal to the valuation of $y$.
When we use a language including the field inverse the intended interpretations are valued fields $F$, and when we replace the field inverse by restricted division the intended interpretations are valuation rings $F^\circ$.
\end{defn}

In the next Remark we comment on the different kinds of division, and how the results relate.
\begin{rem}\label{AQE}In some papers (e.g. \cite{LL2}, \cite{LR2}, \cite{LR3}, and also in Theorems \ref{QED} and \ref{QELR2} below) we have worked in the languages
$\cL_{\cA(K)}^D$ and $\cL_{\cA_{sep}(K)}^D$, in which we omit the function for field inverse and include symbols for restricted division, and in others (e.g. \cite{CL1}) we have worked in the language $\cL_{\cA(K)}$ of the Introduction, in which we have a symbol for field inverse.
We remark that  $\Kalgo$ has quantifier elimination in the language $\cL_{\cA(K)}^D$ (or $\cL_{\cA^H(K)}^D$ or $\cL_{\cA_{sep}(K)}^D$) if and only if $\Kalg$ has quantifier elimination in the language $\cL_{\cA(K)}$ (or $\cL_{\cA^H(K)}$  or $\cL_{\cA_{sep}(K)}$).  We can cover $(\Kalg)^n$ by polydiscs of the form $\{x:  \bigwedge_i ( |x_i|  \Box_i 1) \}$ where the $\Box_i \in \{\leq, \geq \}$, and we can map such a polydisc to
$(\Kalgo)^n$ by sending $x_i \mapsto x_i$ if $\Box_i$ is $\leq$ and $x_i \mapsto (x_i)^{-1}$ if $\Box_i$ is $\geq$.
For $\cF$ any of $\cA(K)$, $\cA^H(K)$ or $\cA_{sep}(K)$, it is clear that
 any $\cL_{\cF}^D$--quantifier-free definable subset of $(\Kalgo)^n$ is $\cL_{\cF}$--quantifier-free definable.  For the converse we can easily show by induction on the number, $k$, of occurrence of symbols from $\bigcup A_{m,n}$ in the formula, that an $\cL_{\cF}$--quantifier-free definable subset of
$(\Kalgo)^n$ is indeed $\cL_{\cF}^D$--quantifier-free definable.  (By the same induction show that
if an $\cL_{\cF}$--term on a quantifier-free $\cL_{\cF}^D$--definable subset $X$ of $(\Kalgo)^n$ is always $\leq 1$ in size, then it is equivalent to an $\cL_{\cF}^D$--term. Actually it is sufficient to just establish this piecewise on a quantifier-free $\cL_{\cF}^D$--definable cover of $X$, but see the proof of Lemma \ref{hf3} for combining terms defined piecewise into one term.)
\end{rem}

Statement (ii) of the following proposition may seem surprising but comes in part from the fact that piecewise definitions of terms on certain pieces can be combined into a single term.

\begin{lem}\label{hf3} \item[(i)] If $K$ is a henselian field, then $\Ko$ satisfies the Implicit Function Theorem:  Given  a system  $P_i(y_1, \cdots , y_n)=0, i=1, \cdots,n$ of $n$ polynomial equations in $n$ unknowns over $\Ko$, and $\overline y = (\overline y_1,\cdots , \overline y_n) \in (\Ko)^n$  such that the $P_i(\overline y) \in \Koo$ and the Jacobian
$ \big|\big(\frac{\partial P_i}{\partial y_j }\big)_ {i,j=1,\cdots , n}\big|(\overline y)$
at the point $\overline y$  is a unit, there is  a unique  $\widetilde y \in (\Ko)^n$ with  $\overline y - \widetilde y \in (\Koo)^n$ satisfying $P(\widetilde y)=0$.
\item[(ii)]
Let $\tau$ be a term of $\cL_{\cA^H}^{D}$.  There is a term $\tau'$ of $\cL_\cA^{D,h}$ such that for every field $K$ with analytic $\cA$--structure, say, via $\sigma$, we have  $\tau^\sigma = \tau'^\sigma.$
\item[(iii)]
Let $\tau$ be a term of $\cL_{\cA^H}$.  There is a term $\tau'$ of $\cL_\cA^{h}$ such that for every field $K$ with analytic $\cA$--structure, say, via $\sigma$, we have  $\tau^\sigma = \tau'^\sigma.$
\end{lem}

\begin{proof}
\item[(i)] \cite{Kuh1} Theorem 24 or \cite{Kuh2} Theorem 9.13.
\item[(ii)] By Lemma \ref{CH} and Lemma-Definition \ref{hf2bis}, we need only prove this for terms $k_\Sigma$ defined by polynomial henselian systems $\Sigma$.  Let $\Sigma$ be given. It follows from \ref{hf2bis} by a standard compactness argument that there are finitely many terms $\tau_i(\xi,\rho)$ of $\cL_\cA^{D,h}$ such that in each henselian field $K$ with analytic $\cA$--structure and for any $(\xi,\rho) \in (\Ko)^m\times(\Koo)^n$ at least one of the $\tau_i$ satisfies $\Sigma$ at $(\xi,\rho)$, that is,   $\tau_i(\xi,\rho) = k_\Sigma^\sigma(\xi,\rho)$.
Let $\phi_i(\xi,\rho)$ be the quantifier-free formula that says that $\tau_i(\xi,\rho)$ satisfies $\Sigma$ and for $j < i, \tau_j(\xi,\rho)$ does not satisfy $\Sigma$.  Observe, by induction on formulas, that for every quantifier-free $\cL_\cA^{D,h}$--formula $\phi$ there is an $\cL_\cA^{D,h}$--term $\chi_\phi$ that gives the characteristic function of $\phi$.  (The characteristic function of $x \neq 0$ is $D(x,x)$ and the characteristic function of $0 \neq |x| \leq |y|$ is $D(D(x,y),D(x,y))$).  Now take $\tau' := \sum \chi_{\phi_i} \cdot \tau_i$.
\item[(iii)] The proof is similar to (ii).
\end{proof}

\subsection{Weierstrass Preparation for $\{A_{m,n}^H\}$} In this section we will show that the
$A_{m,n}^H$ satisfy {\it separated} Weierstrass Preparation (and thus Division).

Our first main result of this paper can be stated as follows with terminology from \cite{CL1} which is recalled in the appendix to this paper. The proof of the theorem boils down to the Weierstrass Preparation statement provided by Proposition \ref{Wprep}.

\begin{thm}\label{MainThm}Let $\cA$ be a strictly convergent $(A,I)$--pre-Weierstrass system (Definition \ref{PWS}).
$\cA^H := \{A_{m,n}^H\}$ is an $(A, I)$--Separated pre-Weierstrass system.
\end{thm}
\begin{proof} We need to verify the properties (1) up to (6) from Definition \ref{sepWS}. Each of properties (1) -- (4) is either immediate or easy.
To prove property (6) it is sufficient, by the results of Section \ref{WPWD} to prove Weierstrass Preparation, which we do in Proposition \ref{Wprep}. Property (5) is an immediate consequence of Weierstrass division in the form of property (6).
\end{proof}

Let us recall the definition of regularity in the case of separated power series.

 \begin{defn}[Separated Regular]
\label{regular}
With the above notation where $A$ is a ring and $I \subset A$ is a proper ideal,  let $f \in A[[\xi_1,\ldots,\xi_m,\rho_1,\ldots,\rho_n]]$, and let $J$ be the
ideal
$$
J:=\big\{\sum_{\mu,\nu} a_{\mu,\nu} \xi^\mu\rho^\nu \in A[[\xi,\rho]] :
a_{\mu,\nu} \in I \big\}
$$
of $A[[\xi,\rho]]$.
 \item[(i)]
$f$ is called {\em regular in $\xi_m$ of degree $d$} when $f$ is
congruent in $A[[\xi,\rho]]$ to a monic polynomial in $\xi_m$ of
degree $d$, modulo the ideal $B_1 \subset A[[\xi,\rho]]$, with
$$
B_1:=J+(\rho)A[[\xi,\rho]].$$
 \item[(ii)]
$f$ is called {\em regular in $\rho_n$ of degree $d$} when $f$ is
congruent in $A[[\xi,\rho]]$ to $\rho_n^d$, modulo the ideal
$B_2 \subset A[[\xi,\rho]]$, with
$$
B_2 := J+(\rho_1,\ldots,\rho_{n-1},\rho_n^{d+1})A[[\xi,\rho]].
$$
\end{defn}

\begin{prop} \label{Wprep}
\item[(i)]Let $f (\xi, \rho) \in A^H_{m,n}$ be regular in $\xi_m$ of degree $s$.  Let $\xi' = (\xi_1, \cdots , \xi_{m-1})$.  There are unique $f_0, \cdots , f_{s-1} \in A^H_{m-1,n}$, and unit
$U \in A^H_{m,n}$ such that
$$
f = U\cdot \big(\xi_m^s + f_{s-1}(\xi',\rho) \xi_m^s + \cdots + f_1(\xi',\rho) \xi_m + f_0(\xi',\rho)\big).
$$
\item[(ii)]Let $f (\xi, \rho) \in A^H_{m,n}$ be regular in $\rho_n$ of degree $s$.  Let $\rho' = (\rho_1, \cdots , \rho_{n-1})$.  There are unique $f_0, \cdots , f_{s-1} \in (I,\rho_1,\cdots,\rho_{n-1})A^H_{m,n-1}$, and unit
$U \in A^H_{m,n}$ such that
$$
f = U\cdot \big(\rho_n^s + f_{s-1}(\xi,\rho') \rho_n^s + \cdots + f_1(\xi,\rho') \rho_n + f_0(\xi,\rho')\big).
$$
\end{prop}
\begin{proof}
We prove (i).  The proof of (ii) is similar. Let $f(\xi,\rho) \in  A^H_{m,n}$ be regular in $\xi_m$ of degree $s$.
 By Proposition \ref{CH}, we may write
 \begin{equation}\label{e1}
 f = F_0(\xi, \rho) + F_1(\xi, \rho, g_1, \cdots, g_M, h_1,\cdots,h_{N})
\end{equation}
where $F_0 \in A[\xi,\rho],$
$F_1(\xi,\rho,\eta,\lambda) \in (I, \rho, \lambda)A_{m+M,n+N}$, and the $g_i, h_j$ are defined by a henselian system $\Sigma(\eta,\lambda)$:
\begin{equation} \label{e32}
\begin{split}
\eta_i &= b_{0,i}(\xi,\rho) + b_{1,i}(\xi,\rho,\eta,\lambda) \quad i=1,\cdots,M \\
\lambda_j &= c_{0,j}(\xi,\rho) + c_{1,j}(\xi,\rho,\eta,\lambda) \quad j=1,\cdots,N
\end{split}
\end{equation}
where the
\begin{eqnarray*}
&b_{0,i} \in A[\xi,\rho], &b_{1,i} \in (I,\rho,\lambda)A_{m+M,n+N} \text{ \quad and the} \\
&c_{0,j} \in (\rho)A[\xi,\rho], &c_{1,j} \in (I,\rho,\lambda^2)A_{m+M,n+N}.
\end{eqnarray*}
Let $\xi' = (\xi_1, \cdots , \xi_{m-1})$.

We follow the outline of the proof of \cite{Ar}.
Let
$a_0, \cdots , a_{s-1}$
be new variables and let
$$
a:= \xi_m^s + a_{s-1}\xi_m^{s-1} + \cdots + a_0.
$$
Let $g^*_{i,k}, h^*_{j,k}$ be new variables,
$i = 1,\cdots,M; j= 1, \cdots ,N; k=0,\cdots,s-1$, and let
\begin{eqnarray*}
g_i^* := \sum_{k=0}^{s-1} g^*_{i,k}\xi_m^j  \quad \text{ and } \quad h_j^* := \sum_{k=0}^{s-1} h^*_{j,k}\xi_m^j.
\end{eqnarray*}

We seek a henselian system that will determine  $a_0, \cdots , a_{s-1}$ as the Weierstrass data of $f$ and
the $g^*_{i,k}$, (respectively $h^*_{j,k}$) as the Weierstrass data of the remainder on dividing $g_i$ (respectively $h_j$) by $f$.

In equation (\ref{e1}), setting $f = 0$ and replacing the $g_i$ by $g_i^*$ and the $h_j$ by $h_j^*$
(i.e. working modulo $a$) we get
 \begin{equation}\label{e2}
 0 = F_0 + F_1(\xi', \rho, g_1^*, \cdots, g_M^*, h_1^*,\cdots,h_{N}^*).
 \end{equation}
By strictly convergent Weierstrass division by $a$, which is (strictly convergent) regular in $\xi_m$ of degree $s$ we get equations
 \begin{equation}\label{e3}
 0 = \sum_{k=0}^{s-1} H_k(\xi', \rho, \{a_k\}, \{g_{ik}^*\}, \{h_{jk}^*\})\xi_m^j,
 \end{equation}
where we have written $\{a_k\}$ to denote the variables $a_0,\cdots,a_{s-1}$, and so on.
From the fact that $f$ (and hence also $F_0 = F_0(\xi,\r)$) is regular of degree $s$ in $\xi_m$ and that modulo $a$ we have $\xi_m^s = - a_{s-1}\xi_m^{s-1} - \cdots - a_0$, we see directly, equating the coefficients of $\xi_m^i$, for  $i= 0, \dots, s-1$, to $0$ that equation (\ref{e3}) can be rewritten as a henselian system determining $a_0, \cdots, a_{s-1}$ in terms of the other variables $\xi',\rho, \{g_{ik}^*\}, \{h_{jk}^*\}$.

Let $\Sigma(\eta,\lambda)$ be the henselian system determining the $g_i$ and $h_j$.  Divide the equations $\Sigma(\xi, \rho, g_1^*, \cdots, g_M^*, h_1^*,\cdots,h_{N}^*)$
(i.e. the system obtained from $\Sigma(\eta,\lambda)$ by replacing $\eta_i$ by $\sum_k g_{ik}^*\xi_m^k$ and $\lambda_j$ by $\sum_kh_{jk}^*\xi_m^k$)  by $a$
to get a system of the form
\begin{equation}\label{e4}
\begin{split}
g_{i}^* &=\sum_{k=0}^{s-1} G_{ik}(\xi',\rho,\{a_k\}, \{g_{ik}^*\}, \{h_{jk}^*\})\xi_m^k \quad i=1,\cdots,M\\
h_j^* &= \sum_{k=0}^{s-1} H_{jk}(\xi',\rho,\{a_k\}, \{g_{ik}^*\}, \{h_{jk}^*\})\xi_m^k \quad j=1,\cdots,N.
\end{split}
\end{equation}
Equating the coefficients of like powers of $\xi_m$ this yields a henselian system $\Sigma^*$ of the form
\begin{equation}\label{e5}
\begin{split}
g_{ik}^* &= G_{ik}(\xi',\rho,\{a_k\}, \{g_{ik}^*\}, \{h_{jk}^*\}) \quad i=1,\cdots,M\\
h_{jk}^* &= H_{jk}(\xi',\rho,\{a_k\}, \{g_{ik}^*\}, \{h_{jk}^*\}) \quad j=1,\cdots,N.
\end{split}
\end{equation}

Equations  \ref{e5}, and the above described henselian system for $a_0, \cdots , a_{s-1}$ together form the required henselian system for the $ g_{ik}$ and $h_{jk}$, the Weierstrass data of $f$ and the remainders on dividing the $g_i$ and $h_j$ by $f$.  Lemma \ref{IFT}(ii)  shows that these two henselian systems can be combined in one system determining all the variables $a_0, \cdots , a_{s-1},  g^*_{ik}$ and $h^*_{jk}$.

Finally we must show that $U \in A_{m,n}^H$.  Let the $\Theta_i, \Theta'_j$ be new variables, and write
\begin{eqnarray*}\label{e6}
\eta_0 &:=& \Theta_0 \cdot a\\
\eta_i &:=&  \Theta_i \cdot a + \Sum_k g_{ik}^* \xi_m^k \\
\lambda_j &:=&  \Theta'_j \cdot a + \Sum_k h_{jk}^* \xi_m^k .
\end{eqnarray*}
Making this substitution in equations (\ref{e1}) and (\ref{e32}),  doing Weierstrass division by $a$ and using equations (\ref{e5}) and (\ref{e6}) yields equations of the form
\begin{eqnarray*}\label{e7}
a \cdot \Theta_0 &=& a \cdot G_0(\xi, \rho, \{\Theta^*\}, \{a_k\}, \{g_{ik}^*\}, \{h_{jk}^*\})\\
a\cdot \Theta_i  &=&  a \cdot G_i(\xi, \rho, \{\Theta^*\}, \{a_k\}, \{g_{ik}^*\}, \{h_{jk}^*\})\\
a \cdot \Theta'_j &=&   a \cdot G'_j(\xi, \rho, \{\Theta^*\}, \{a_k\}, \{g_{ik}^*\}, \{h_{jk}^*\}),
\end{eqnarray*}
where $\Theta^* = ( \Theta_i, \Theta'_j)$.  The required system of henselian equations for $\Theta_0$ and the $\Theta^*$ over the other variables is then
\begin{eqnarray*}\label{e8}
\Theta_0  &=&  G_0(\xi, \rho, \{\Theta^*\}, \{a_k\}, \{g_{ik}^*\}, \{h_{jk}^*\})\\
\Theta_i  &=&   G_i(\xi, \rho, \{\Theta^*\}, \{a_k\}, \{g_{ik}^*\}, \{h_{jk}^*\})\\
\Theta'_j  &=&   G'_j(\xi, \rho, \{\Theta^*\}, \{a_k\}, \{g_{ik}^*\}, \{h_{jk}^*\}),
\end{eqnarray*}
and we can again use Lemma \ref{IFT}(ii) to combine the systems.
\end{proof}

\begin{rem}\label{div}
It follows from Weierstrass Division that if $f \in A_{m,n}^H$ and we have $f(\xi',0,\rho) = 0$ then there is a
 $g \in A_{m,n}^H$ such that $f = \xi_m \cdot g$.  (And similarly w.r.t. the other variables).
\end{rem}

With a suitable definition of ``regular'', the analogue of Proposition \ref{Wprep} holds for the terms $k_\Sigma$. (Of course uniqueness holds only at the level of $k_\Sigma^\sigma$).

In order to be able to extend any $\cA$-structure on any field $K$ to a $\cA^H$-structure on $K$, we introduce the following condition of goodness.

\begin{defn}\label{good}Let $\cA$ be a strictly convergent pre-Weierstrass system. We call $\cA$ a {\it good pre-Weierstrass system} if for every system $\Sigma$ of henselian equations we have that if $F \in A_{m+M,n+N}$ and $F\circ h_\Sigma = 0$ (as a power series) then $F \in I_\Sigma \cdot ({A_{m+M,n+N}})_{1 + (I, \rho, \lambda)}$.
\end{defn}

\begin{lem}\label{recap0} Let $\cA$ be a good strictly convergent pre-Weierstrass system, and let $\sigma$ be an analytic $\cA$-structure on a Henselian valued field $K$.
Then there exists a unique separated $\cA^H$-structure $\sigma^H$ on $K$ which extends $\sigma$ (and which often will be also denoted by $\sigma$).
Precisely, if a series $f\in A_{m,n}^H$ is a composition $f=g(\xi,\rho,h_\Sigma(\xi,\rho))$ as in Definition \ref{compseries},
then $\sigma^H(f)$ is the function from $(K^{\circ})^{m+M}\times (K^{\circ\circ})^{n+N}$ to $K$ given by
$$
F^\sigma \circ k_\Sigma^\sigma,
$$
with notation from Lemma-Definition \ref{hf2bis}.
\end{lem}
\begin{proof}
Lemma-Definition \ref{hf2bis} and Definition \ref{good} guarantee that the separated $\cA^H$-structure $\sigma^H$ on $K$ is well defined, that is, if  $g\circ h_\Sigma = 0$ then there is a unit $u \in {A_{m+M,n+N}}_{1 + (I, \rho, \lambda)}$ such that $u\cdot g \in I_\Sigma$ and hence $g^\sigma \circ k_\Sigma^\sigma = 0$.
\end{proof}

The following result will be useful in sections \ref{ex1a} and \ref{ex2a} .

\begin{prop}[Criterion]\label{complete}
Assume strictly convergent pre-Weierstrass system $\cA= \{A_{m}\}$ has the  property that
there is a separated Weierstrass system $\cB= \{B_{m,n}\}$ such that for all $(m,n)$ we have $A_{m,n} \subset B_{m,n}$ and $B_{m,n}$ is faithfully flat over $(A_{m,n})_{1 + (\rho)A_{m+n}}$.  Then
\item[(i)] $\cA$ is a good pre-Weierstrass system, and
\item[(ii)] $\cA^H$ is a separated Weierstrass system.
\end{prop}

\begin{proof}
(i) From \cite{Mat} Theorem 7.5 it follows that $I_\Sigma \cdot (A_{m,n})_{1 + (\rho)A_{m,n}} =  (I_\Sigma \cdot B_{m+M,n+N}) \cap (A_{m,n})_{1 + (\rho)A_{m,n}}$.   Hence if $F \circ h_\Sigma = 0$ in $B_{m,n}$ then $F \in I_\Sigma \cdot (A_{m,n})_{1 + (\rho)A_{m,n}}$.

For (ii) observe that property (vii)$'$ of Definition \ref{SNPstr'} (from which property (vii) follows) can be expressed as a linear equation over $A_{m,n}^H$ which has a solution in $B_{m,n}$.  We can capture the condition $g_{\mu\nu} \in B_{m,n}^\circ$
by writing $g_{\mu\nu} = \sum_i c_i g_{\mu\nu i} + \sum_j \rho_j g'_{\mu\nu j}$ for suitable $c_i \in I$,
and the $g_{\mu\nu i},g'_{\mu\nu j} \in B_{m,n}$. Let $f = F\circ h_\Sigma$ and
$\overline f_{\mu\nu} = F_{\mu\nu} \circ h_{\Sigma_{\mu\nu}}$ with $F, F_{\mu\nu} \in A_{p,q}$ for suitable $p,q$.
Let $\Sigma'$ be the conjunction of the henselian systems $\Sigma$ and the
$\Sigma_{\mu\nu}$.  Then the linear equation
\begin{equation}
F =    \sum_{(\mu,\nu) \in J}^{} F_{\mu\nu}(\xi'')^\mu(\rho'')^\nu\big(1 +  \Sum_i c_i g_{\mu\nu i} + \Sum_j\rho_jg'_{\mu\nu j}\big) \mod (I_{\Sigma'})
\end{equation}
has coefficients from $A_{p,q}$ and has a solution in $B_{p,q}$, for suitable $(p,q)$, and hence in $(A_{p,q})_{1 + (\rho^*)A_{p,q}}$.
Substituting $h_{\Sigma'}$ for the appropriate variables gives us (vii)$'$.
\end{proof}

Lemma \ref{hf2bis} yields the following corollary.
\begin{cor}\label{sim} Let $\cA$ be a good pre-Weierstrass system and let $f \in \bigcup_{m,n} A_{m,n}^H$. Let $\sigma$ be an analytic $\cA$--structure on a Henselian valued field $K$. Write $\sigma$ for the $\cA^H$-structure on $K$ given by Lemma \ref{recap0}.
The graph of $f^\sigma$ is existentially $\cL_\cA$-definable. Moreover, the  existential formula defining the graph of $f^\sigma$ can be taken to be independent of the analytic structure $\sigma$.
\end{cor}

Now we come to  two other main results in the following two subsections, which cover and extend a number of examples considered in previous papers  \cite{vdD}, \cite{LR2}, \cite{LR3}, \cite{CLR1}, \cite{CL1}. In subsection \ref{ex}, we introduce a new general concept of analytic structures and provide some further examples, complementing theorems \ref{ex1} and \ref{ex3}.
In Section \ref{Cex} we show that the extension of $\cA$ to $\cA^H$ is necessary for quantifier elimination.

\subsection{The case of $K$ a complete field and $A_m = T_m(K)^\circ$}\label{ex1a}
In this subsection we discuss the ``classical'' example.
Let $K$ be a complete, rank one  valued field (not necessarily algebraically closed) and let $A_m := \Ko\langle\xi\rangle = T_m(K)^\circ,$ the ring of strictly convergent power series over $\Ko$ and let $\cA:=\{A_m\}$.  This is certainly a pre-Weierstrass system.  Any complete, rank one  field $F \supset K$ has analytic $\cA$--structure. In this case, $A = \Ko$ and $I = \Koo$.

\begin{thm} \label{ex1}
With the above notation, $\cA$ is a good pre-Weierstrass System and $\cA^H := \{A_{m,n}^H\}$ is a separated Weierstrass System.
\end{thm}

\begin{proof}[Proof of Theorem \ref{ex1}]  The proof is immediate from Proposition \ref{complete}, taking $B_{m,n} = S_{m,n}(K)^\circ$, and using \cite{R} Lemma 3.1 which establishes that $S_{m,n}(K)^\circ$ is faithfully flat over $(T_{m+n}^\circ)_{1+(\rho)T_{m+n}^\circ}$.

\end{proof}

In \cite{LR2} we showed that for $K$  rank one and complete there is a separated Weierstrass system $\cE$ extending $\cA$ all of whose additional functions are uniformly existentially definable in every {\it algebraically closed} field with analytic $\cA$--structure.  We did not make explicit what the additional functions are, nor did the uniform definitions work in non-algebraically closed fields with analytic $\cA$--structure.

In this section we do not assume $K$ to be algebraically closed, nor do we restrict consideration to  algebraically closed fields with analytic $\cA$-structure.
In this larger generality we obtain quantifier elimination in the forms of Theorems \ref{QEL} and \ref{QELrem}, and we make explicit what the additional existentially definable functions are, namely they are defined by terms of
$\cL_{\cA^{H}}$ (or even $\cL_\cA^{h}$),
on {\it all} fields with analytic $\cA$--structure,
and the graphs of the functions $h_\Sigma$ are indeed quantifier-free definable in all such fields.  We also give a result that links definable functions to terms in Theorem \ref{thens}. (The corresponding results in a more abstract framework are bundled in Theorem \ref{QEL3} below).
Similar results hold with $\cL^D$ in place of $\cL$, see Remark \ref{AQE}.

In the following result, $\cL_{\cA^H}$ is the valued field language $\cL$ with function symbols for the elements of the $A^H_{m,n}$ adjoined for all $m$ and $n$.
\begin{thm}\label{QEL}
Let $L$ be any algebraically closed valued field with strictly convergent analytic $\cA$-structure. Then $L$ allows quantifier elimination in both $\cL_{\cA^H}$ and $\cL_\cA^{h}$ (and thus quantifier simplification   (i.e. every formula is equivalent to an existential formula) in $\cL_\cA$).
Moreover, the theory of algebraically closed valued fields with strictly convergent analytic $\cA$--structure eliminates quantifiers in both the languages $\cL_{\cA^H}$ and $\cL_\cA^{h}$ (both considered as the intended definitial expansions of $\cL_\cA$).
\end{thm}
\begin{proof}
Follows from Theorem \ref{ex1}, Lemma \ref{hf3} (iii), Theorem 4.5.15 of \cite{CL1} for $\cA^H$, amended by Remark \ref{correction0}.
\end{proof}

\begin{thm}\label{terms}
(i) (characteristic $0$)
Let $x = (x_1,\dots,x_n)$ be several variables and $y$ one variable. Suppose that a formula $\varphi(x,y) \in \cL_\cA$ defines $y$ as a function of $x$ in all algebraically closed fields of characteristic $0$ with analytic $\cA$--structure. Then there is a term $\tau$ of $\cL_\cA^{h}$ such that
$L \models \varphi(x,y) \leftrightarrow y = \tau(x)$ for all algebraically closed valued fields $L$ of characteristic $0$ with analytic $\cA$--structure.

(ii) (characteristic $p \neq 0$) In characteristic $p>0$ let $\cL^{'}_\cA$ be the language $\cL_\cA$ with the $p$-th root function $(\cdot)^\frac{1}{p}$ adjoined.  In characteristic $p$ the analogous result to (i) holds with $\cL^{'h}_\cA$ in place of $\cL_\cA^h$. Namely, for any $\cL_\cA$-definable function $f$ for the theory $T$ of algebraically closed valued fields of characteristic $p$ with analytic $\cA$--structure, defined on a Cartesian power of the valued field, there exists an $\cL^{'h}_\cA$-term $t$ such that $T$ proves that $t$ and $f$ are the same function.
\end{thm}
\begin{proof}
Follows from Proposition \ref{termsalg} and Lemma \ref{hf3} (iii).
\end{proof}

Similarly, but invoking Theorem 6.3.7 instead of Theorem 4.5.15 of \cite{CL1}, we find a more general but more involved quantifier elimination result, using leading terms sorts. Let $\cL_{\Hens,\cA}$ (resp. $\cL_{\Hens,\cA^H}$) be as in Section 6.2 of \cite{CL1}, namely the multi-sorted language which has the field language and function symbols for all $f$ in $A_m$ and all $m$ (resp. in  $A_{m,n}^H$ for all $m,n$) on the valued field sort, the maps $rv_N$ for each integer $N>0$ and the full induced structure on the sorts $RV_N$. Recall that on a valued fiell $L$, $RV_N(L)$ is the quotient of multiplicative semi-groups $L/1+nL^{\circ\circ}$, and $rv_N:L\to RV_N(L)$ the projection.

\begin{thm}\label{QELrem}
The theory of characteristic zero henselian valued fields with strictly convergent analytic $\cA$-structure eliminates the valued field quantifiers in both the languages $\cL_{\Hens, \cA^H}$ and $\cL_{\Hens, \cA}^{h}$ (both considered as the intended definitial expansions of $\cL_{\Hens,\cA}$).
\end{thm}

We can also formulate the following result on the term-structure of definable functions. To this end, let $\cL_{\Hens,\, \cA}^*$, (resp.~$\cL_{\Hens,\cA^H}^{*}$)  be the language $\cL_{\Hens,\, \cA}$, (resp.~$\cL_{\Hens,\cA^H}$) together with all the functions $h_{m,n}$ as defined in Definition 6.1.7 of \cite{CL1}. Note that these functions $h_{m,n}$ are multi-sorted variants of the basic henselian functions $h_n$ of Definition \ref{hf1}.

\begin{thm}[Term structure]\label{thens}
The theory of characteristic zero henselian valued fields with strictly convergent analytic $\cA$-structure has the following property, uniformly in its models $K$.
Let $X \subset K^n$ be
definable and let $f:X\to K$ be an $\cL_{\Hens,\cA}$-definable
function. Then there exist an
$\cL_{\Hens,\cA}$-definable function $g:X\to RV_N(K)^N$ for some $N> 0$ and an $\cL_{\Hens,\cA}^{*}$-term $t$ such that
\begin{equation}\label{et}
f(x)=t(x,g(x))
\end{equation}
for all $x\in X$.
\end{thm}
\begin{proof}
By Theorem 6.3.8 of \cite{CL1}, we can find $g$ and a term $t$ in $\cL_{\Hens,\cA^H}^{*}$ such that (\ref{et}) holds for all $x\in X$. By compactness and using characteristic functions as at the end of the proof of Lemma \ref{hf3}, we see that such $g$ and $t$ can be taken independently from the choice of model $K$. By Lemma \ref{hf3}(iii), we can take $t$ to be in $(\cL_{\Hens,\cA}^{*})^h$, using the same $g$. Now, if we change $g$ we can rephrase the $h_n$ in terms of the $h_{m,n}$, to conclude the proof.
\end{proof}

\begin{rem}
Note that the elimination of valued field quantifiers of Theorem \ref{QELrem} also holds for $\cL_{\Hens, \cA}^{*}$ instead of $\cL_{\Hens, \cA}^{h}$, by the argument at the end of the proof of Theorem \ref{thens}.
\end{rem}

Because henselian functions, reciprocals and $D$--functions are the unique solutions of their defining equations, with the above notation, we have the following corollary.  The result analogous to (i) implicit in \cite{LR2} gave only a finite-to-one projection. In \cite{LR6} we proved the result analogous to (ii) for $\cL_{\cA_{sep}}$--subanalytic sets.

\begin{cor} (i) Every $\cL_\cA$--subanalytic set (i.e. a subset of $\Kalg^{~n}$ defined by an $\cL_\cA$--formula) is the one-to-one projection of an $\cL_\cA$--semianalytic set (i.e. a subset of $\Kalg^{~n+N}$ defined by a quantifier-free $\cL_\cA$--formula which does not involve $(\cdot)^{-1}$). \\
(ii) Every $\cL_\cA$--subanalytic set is the finite disjoint union of $\cL_\cA$--subanalytic manifolds.
\end{cor}

\begin{rem}There are several different definitions of (rigid) semi-analytic sets (cf. \cite{LL2}, \cite{LR1}, \cite{Martin1}).  The above definition is perhaps the most restrictive, and best called ``globally affinoid (or strictly convergent) semi-analytic".
\end{rem}

\subsection{The case when $A$ is a noetherian ring, complete in its $I$--adic topology.}\label{ex2a}

We now treat the case that $A$ is a noetherian ring which is complete in its $I$--adic topology (i.e. $A = A~\widehat{}~ = \varprojlim_n A/I^n$, so $A$ is also $I$--adically separated). This case complements the study of \cite{CLR1}.

For example, one could take $A = \bZ[[t]]$ and
$$
A_m= \bZ[[t]]\langle \xi_1, \cdots , \xi_m \rangle =   \bZ[[t]] [ \xi_1, \cdots , \xi_m ] \ \widehat{}~,
$$
 the $t$--adic completion of $ \bZ[[t]] [ \xi_1, \cdots , \xi_m ]$. Then $\cA := \{A_m\}$ is a strictly convergent pre-Weierstrass system.
  The fields $\bQ_p, \bC_p$ for $p$ prime, and all ultraproducts of these fields,  have analytic $\cA$--structure
via the natural maps $t \mapsto p$.  The fields $\bF_p((t)),$ for $p$ prime, and all ultraproducts of these fields,  have analytic $\cA$--structure via the natural maps $p \mapsto 0$.  We could consider all of these to be fields with analytic structure satisfying the side conditions that for all $p$ either $p = 0$ or $|p|=1$ or $p$ is prime, and that $t$ is prime.  This situation is discussed in Appendix \ref{side}.

There are many more ways that a henselian field can have analytic $\cA$--structure.  For example, the field of puiseux series
$\bC((t^{\frac{1}{\infty}}))$ has analytic $\cA$--structure, via $t \mapsto t$. As another example, let $\bC_p^*$ be a non-principal ultrapower of $\bC_p$ and consider the mapping $\sigma$ that sends $t $ to $[p, p^{\frac{1}{2}}, p^{\frac{1}{3}}, \cdots]$, so in
$\bC_p^*$ the element $\sigma(t)$ is an ``infinitesimal" power of $p$.  This $\sigma$ gives an analytic $\cA$--structure on $\bC_p^*$. Alternatively, one could map $t$ to a finite but irrational power of $p$.

Theorem \ref{ex3}  below explains, via \cite{CL1}, aspects of the geometry and model theory of definable sets in all these analytic structures uniformly.

Note that Theorem \ref{ex1} above is not covered by Theorem \ref{ex3}, except when $K$ is discretely valued.

\begin{thm}\label{ex3} Let $A$ be a noetherian ring, $I$ an ideal of $A$ with $I\not=A$, and assume that $A$ is complete in its $I$--adic topology.  Let $A_m := A[\xi_1, \cdots ,\xi_m]~\widehat{}~$ be the $I$--adic completion of $A[\xi_1, \cdots ,\xi_m]$.  Then  $\cA$ is a good pre-Weierstrass System and $\cA^H := \{A_{m,n}^H\}$ is a separated Weierstrass System.
\end{thm}

\begin{proof}
 We showed in \cite{CLR1} Lemma 2.9 that if we define
$$
B_{m,n} :=A\langle\xi\rangle[[\rho]] := A[\xi_1, \cdots ,\xi_m] ~\widehat{}~[[\rho_1, \cdots , \rho_n]]
$$
and
$\cB := \{B_{m,n}\}$, then $\cB$ is a separated Weierstrass system. Indeed we showed that these rings satisfy condition (7)$'$ of Definition \ref{SNPstr'}.
Hence the result follows from Proposition \ref{complete} and the fact that $B_{m,n}$ is faithfully flat over $(A_{m,n})_{1 +(I,\rho)A_{m,n}}$.  Since the $(\rho)$--adic completion of $(A_{m,n})_{1 +(I,\rho)A_{m,n}}$ is $B_{m,n}$, faithful flatness follows from  \cite{Mat} Theorem 8.14.
\end{proof}

We have the following quantifier elimination result, as in Section \ref{ex1a}.

\begin{thm}\label{QEL2}
Theorems \ref{QEL}, \ref{terms}, \ref{QELrem}, and \ref{thens} go through in literally the same way, but with $\cA$ as in this Section \ref{ex2a} instead of as in Section \ref{ex1a}.
\end{thm}

\subsection{A general concept and examples}\label{ex}\label{formal}

Motivated by Theorems \ref{ex1} and \ref{ex3}, we introduce the following terminology, of which these theorems in fact provide examples. See the appendix for a summary and elaboration of some of the definitions of \cite{CL1}.

\begin{defn}\label{SNPstr0} We call a strictly convergent \textbf{\emph{pre}}-Weierstrass system $\cA = \{A_m\}$ (Definitions \ref{PWS} and \ref{good}) a {\it a strictly convergent Weierstrass system}
 if $\cA$ is good and $\cA^H$ is a separated $(A,I)$-Weierstrass system (as defined in \cite{CL1} and also in Definitions \ref{sepWS} and \ref{con7} in the Appendix below).
\end{defn}

By the very nature of Theorems \ref{ex1} and \ref{ex3}, any analytic structure $\cA$ as in these theorems is  a strictly convergent Weierstrass system (as opposed to just a {\textbf{\emph{pre}}}-Weierstrass systems).

Before giving further examples we  outline an alternate proof for Theorem \ref{ex1}.  Let $A_m = T_m(K)^\circ$ and
$f = F\circ h_\Sigma \in A_{m,n}^H \subset S_{m,n}(K)$.  Let $|c| = \|f\|$ be the gauss-norm of $f$ in
$S_{m,n}(K)$.  We allow the possibility that $c=0$. Let $F$ and $h_\Sigma$ be defined over the
$B$--ring $E$.  If $|c| < 1$ we have that $\widetilde F \in \widetilde I_\Sigma \cdot \widetilde K[\xi][[\rho]]$, and hence by Krull's Theorem (\cite{Mat} Theorem 8.10),
also $\widetilde F \in \widetilde I_\Sigma \cdot \widetilde K[\xi][\rho]$.
By the usual induction on the levels of $E$ we eventually find $c \in \Ko$ and $G$ such that $f = c \cdot G \circ h_\Sigma$ and
$F - c \cdot G \in I_\Sigma$.  The case $c = 0$ gives ``goodness'', and the case
$c \neq 0$ reduces proving (vii) or (vii)$'$ to the special case that $\|f\| = 1$.
That case can be dealt with by induction on $m+n$, by breaking up $f$ into pieces exactly as in \cite{CL1} Remark 4.1.11(ii).

\begin{thm}\label{QEL3}
Theorems \ref{QEL}, \ref{terms}, \ref{QELrem}, and \ref{thens} go through in literally the same way, but with $\cA$ any strictly convergent Weierstrass system as in Definition \ref{SNPstr0} instead of more the concrete structure of Section \ref{ex1a}.
\end{thm}

Using Theorem \ref{ex3}, it is easy to write down many examples of strictly convergent Weierstrass systems other than those discussed at the beginning of the previous subsections. Here are yet some further examples of a different kind.

\begin{itemize}
\item [a)](cf. \cite{CL1} Section 4.4(10)).  Let $K$ be complete, rank one and let $\cA = \{A_m\}$ where
  $A_m = \Ko \langle \langle \xi_1, \cdots , \xi_m \rangle \rangle$, the ring of {\it overconvergent} power series
  in $(\xi_1, \cdots , \xi_m)$.  Then $\cA= \{A_m\}$ is a good pre-Weierstrass System and $\cA^H := \{A_{m,n}^H\}$ is a separated Weierstrass System. The proofs of these statements are similar to the proof of Theorem \ref{ex1}.
  \item[b)](cf. \cite{CL1} Section 4.4(4)).  Let $F$ be a {\it maximally complete field} and let
  $A_m = T_m(\cB)$ be the full ring of strictly convergent power series as defined in loc. cit.  Then
  $\cA = \{A_m\}$ is a good pre-Weierstrass System and $\cA^H := \{A_{m,n}^H\}$ is a separated Weierstrass System. The proof is  similar to the alternate proof of Theorem \ref{ex1} outlined above, except that transfinite induction on the support of $E$ is used.
 \item[c)] (cf. \cite{CL1} Section 4.4(7)).  Let $L$ be a field and let $$A_m = \bigcup_iL[[x_1, \cdots , x_i]]\langle \xi_1, \cdots , \xi_m \rangle $$ where $L[[x_1, \cdots , x_i]]\langle \xi_1, \cdots , \xi_m \rangle$ is the $(x_1, \cdots , x_i)$--adic completion of \\  $L[[x_1, \cdots , x_i]] [ \xi_1, \cdots , \xi_m ]$.  Then $\cA = \{A_m\}$ is  a good pre-Weierstrass System and $\cA^H := \{A_{m,n}^H\}$ is a separated Weierstrass System.  This is immediate from Theorem \ref{ex3}, since $A_m$ is the direct limit of rings covered by that theorem.
\end{itemize}

Several results for general strictly convergent Weierstrass systems now follow from  \cite{CL1}
such as quantifier elimination results similar to Theorem \ref{QEL2}, and the cell decomposition and Jacobian property of Theorem 6.3.7 of \cite{CL1}.

\begin{rem}
We also observe that the above methods allow one to give a sharpening of the results of \cite{CL1}, Section 4.6.  Let $K$ be a rank $1$ henselian field, and let $A(K)_{alg, m,n}$ be the $K$--algebra of all power series in $(\xi_1, \cdots, \xi_m, \rho_1, \cdots, \rho_n)$ that are defined by polynomial henselian systems over $(\Ko,\Koo)$. Let $\cA_{alg, str} := \{A(K)_{alg, m,0}\}_m$ and $\cA_{alg, sep} := \{A(K)_{alg, m,n}\}_{m,n}$. Then

\begin{prop}$\cA_{alg, str}$ is the smallest $(\Ko, \Koo)$--strictly convergent Weierstrass system and
$\cA_{alg, sep}$ is the smallest $(\Ko, \Koo)$--separated Weierstrass system.
\end{prop}
\end{rem}

\section{The Counter Example} \label{Cex}
The quantifier elimination of Denef and van den Dries \cite{DD} for $K = \bQ_p$ in the strictly convergent analytic language $\cL_\cA^D$ (see definitions below) raised the obvious question of whether a similar result held for $\bC_p$, the completion of the algebraic closure of $\bQ_p$ (or in general for complete, rank one,
algebraically closed valued fields) since there is quantifier elimination for algebraically closed valued fields in the {\it algebraic} language.  In \cite{S1}, \cite{S2} Schoutens gave such a quantifier elimination in the more restrictive language with function symbols for the functions defined by {\it overconvergent} power series, and developed the corresponding theory of {\it strongly subanalytic sets}.  In \cite{LL2} a quantifier elimination
was given in a language ($\cL_{\cA_{sep}}^D$) with function symbols for a larger class of functions  --  those defined by the so called separated
power series.  This led to the theory of rigid subanalytic sets which was developed further in \cite{LR1}  --  \cite{LR4},  \cite{LR6},
\cite{CLR1},  \cite{CL1},  and \cite{Ce1} --  \cite{Ce5}.  An elaborate proof of quantifier elimination in the strictly convergent language, based on a ``Flattening Theorem'', was published, but the proof was not correct. (See \cite{LR5} for the history and a counterexample to the ``Flattening Theorem").  Theorem \ref{cex} shows that there is no quantifier elimination for complete, rank one,  algebraically closed fields of characteristic zero in the strictly convergent language $\cL_\cA^D$ (see definition below) with $\cA = \{T_m(K)^\circ\}$.
Fortunately, this negative result does not lead to wild behavior, as there are Theorem \ref{ex1} and the quantifier elimination results of Theorem \ref{QELR2}.

The results of \cite{LR2} and, in a more explicit form, Theorem \ref{ex1} of this paper show that for $K$ a henselian field $\Kalg$ has quantifier {\it simplification} in the strictly convergent language $\cL_{\cA}$.  We show in this section that there is no quantifier {\it elimination} in the strictly convergent language $\cL_\cA^D$,
and hence neither in the language denoted by $\cL_K$ in the introduction.
To be precise:  Let $K$ be a complete, rank one,  non-trivially valued field of characteristic $0$, and let  $\Kalg$ be the algebraic closure of $K$ (we do not exclude the case that $K$ is itself algebraically closed).  The assumption that $K$ has characteristic $0$ is not essential, but it simplifies the arguments in a number of places. For valued field $F$ let $F^\circ := \{x \in F : |x| \leq 1 \}$  and $F^{\circ\circ}  :=  \{x \in F : |x| < 1 \}$.
$T_m(K)^\circ$ is the ring of strictly convergent power series in $m$ variables with coefficients from $\Ko$  and $S_{m,n}(K)^\circ$ is the ring of separated power series in $m$ variables of the first kind and $n$ of the second kind (cf. \cite{LL2}, \cite{LR1}) with coefficients from $\Ko$.
\begin{eqnarray*}
A_m &:=& T_m(K)^\circ\\
\cA &:=& \{A_m\}_{m \in \bN}\\
A_{m,n} &:=& S_{m,n}(K)^\circ\\
\cA_{sep} &:=& \{A_{m,n}\}_{m,n \in \bN}.
\end{eqnarray*}
For $\cF$ a collection of functions we let $\cL_\cF$ denote the language of valued fields augmented with symbols for the functions in $\cF$, and let $\cL_\cF^D$ denote $\cL_\cF$ with the field inverse replaced by a symbol $D$ for restricted division adjoined (or two such symbols in the separated case).

Before proceeding we recall Remark \ref{AQE} about the equivalence of the languages with full field inverse and languages with only restricted division.

We know
\begin{thmm}(\cite{LL2})\label{QED} $\Kalgo$ has quantifier elimination in $\cL_{\cA_{sep}}^D$.
\end{thmm}

\noindent And, as a special instance of Theorem \ref{QEL}, we have

\begin{thmm}\label{QELR2}
$\Kalgo$ has quantifier simplification in $\cL_\cA$  and quantifier elimination in both $\cL_{\cA^H}^D$ and $\cL_\cA^{D,h}$.
\end{thmm}

\noindent The counterexample to the question described in the introduction is provided by the following

\begin{thmm}\label{cex} $\Kalgo$ does not have quantifier elimination in $\cL_\cA^D$. Indeed there is a strictly convergent subanalytic set $X \subset (\Kalgo)^3$ (i.e. a set described by an existential $\cL^D_\cA$--formula) that is not described by any quantifier-free $\cL_\cA^D$--formula (i.e. is not $\cL_\cA^D$--semianalytic).
\end{thmm}

The rest of this section is devoted to the proof of Theorem \ref{cex}.

Let $\Kalg^*$ be a nonprincipal ultrapower of $\Kalg$, and
let $\mathfrak p ^{\infty} = \{a \in (\Kalg^*)^\circ : |a| < \frac{1}{n} \text { for all } n \in \bN\}$  be the ideal of infinitesimals in
$(\Kalg^*)^\circ$.  Let $K_1 = Q((\Kalg^*)^\circ /  \mathfrak p ^{\infty})$ be the field of fractions
of $(\Kalg^*)^\circ /  \mathfrak p ^{\infty}$.  By the results of \cite{LR3} and \cite{CL1} we know that the fields $\Kalg$, $\Kalg^*$ and $K_1$ all are algebraically closed valued fields with analytic $\cA$--structure, and also analytic $\cA_{sep}$--structure.  Hence these fields are elementarily equivalent in each of the languages $\cL_\cA$, $\cL_\cA^D$, $\cL_\cA^{D,h}$, and $\cL_{\cA_{sep}}^D$.  Fix $p \in K$ with $0 < |p| < 1$, and let
\begin{eqnarray*}
f(\zeta) := \sum_{n=1}^\infty {p^n \zeta^{n!}},
\end{eqnarray*}
so $f$ is strictly convergent, but definitely not overconvergent.  Let $X \subset (K_1^\circ)^3$ be defined by the formula
\begin{eqnarray}
(\exists z)\big [ |\rho| < 1 \wedge ~ \xi_1 + \xi_2 z = f(z) ~ \wedge ~ \rho z^2 - z +1 = 0 ~ \wedge ~  |z| \leq 1  \big ]
\end{eqnarray}
and for $\br \in K_1^\circ$ let $X(\br) := \{(\br,\1,\2) : (\br,\1,\2) \in X\}$.

For $|\rho| < 1$ let $z_1 = z_1(\rho)$ be the zero of $\rho \zeta^2 - \zeta +1$ with $|z| \leq 1$, and
let $z_2 = z_2(\rho)$ be the other zero (which is of size $| \frac{1}{\rho}|$).  We shall show that there
are $ \br, \3, \4 \in K_1^\circ$ such that $X$ is not described by a quantifier-free $\cL_\cA^D$--formula
in any $K$-rational domain $\cU \subset (K_1^\circ)^3$ with $ (\br, \3, \4) \in \cU$.  This will show that
$X$ is not $\cL_\cA^D$--semianalytic.  Observe that $X$ is an $\cL_{sep}$-- (indeed $\cL_\cA^h$-- ) analytic  variety in the open disc $|\r| < 1$.

Let $\br \in K_1$ satisfy $1 - \frac{1}{n} < |\br| < 1$ for all $n \in \bN$, and define $\3, \4 \in K_1^\circ$ by the equations
\begin{equation}\label{xidef1}
\begin{split}
\3 + \4 z_1(\br) &= f(z_1(\br))\\
\3 + \4 z_2(\br) &= f(z_2(\br)).
\end{split}
\end{equation}

\begin{remm}\label{rem}
Note that while $f(z_2(\br))$ as an element of $K_1$ or $\Kalg^*$ is {\it not} defined by the analytic $\cA$--structure
since $|z_2| > 1$.  However,  $f(z_2(\br))$ is well defined as an element of $K_1^\circ$ by the completeness of $K_1$.  Furthermore, for each $n \in \bN$ there are $\1 ^{(n)} = \1 ^{(n)}(\r) \text{ and } \2 ^{(n)} = \2 ^{(n)}(\r) \in \frac{1}{\r^{n!}} K[\r]$ such that if $\overline \xi_i^{(n)} := \xi_i^{(n)}(\br)$ then
$|\overline \xi_i  - \overline \xi_i^{(n)}| < |p ^n|$.  Solve equations (\ref{xidef1})  for $\3, \4$ using that
$z_1+z_2 = z_1z_2 = \frac{1}{\r}$.  Indeed, $\3, \4$ are just the ``Weierstrass data'' evaluated at $\br$ on formally dividing $f(\zeta)$ by $\zeta^2 - \frac{1}{\r} \zeta +  \frac{1}{\r}$.
Let $\xi_i(\r) := \lim_{n \to \infty} \xi_i^{(n)}(\r) \in \Ko[[\frac{1}{\r}]]$.  Direct calculation shows that $\1(\r)$ and $\2(\r)$ are strictly convergent power series in $ \frac{1}{\r}$ that do not converge on any annulus of the form
$1 - \frac{1}{n} \leq \r \leq 1$.  In other words, they converge for $|\frac{1}{\r}| \leq 1$ but not on any bigger disc $|\frac{1}{\r}| \leq 1 + \delta$, $0 < \delta \in \bR$.
\end{remm}

We will need the following definitions and lemmas.

\begin{defnn}
 \item[(i)] Let $\overline \alpha = (\alpha_1, \cdots ,
 \alpha_\ell) \in \Kalgo \setminus \Kalgoo$ and let $\cW_n(\overline \alpha) $ be the $K$--rational domain (i.e. rational domain defined over $K$, see Remark \ref{rem2}(ii) below)
\begin{equation*}
\begin{split}
 \big\{(\r,\1,\2) : |p| \leq |\r^{n!}| \leq 1 \wedge \bigwedge_{i=1}^\ell |\r - \alpha_i| = 1
\wedge |\1 -\1^{(n)}&(\r)| \leq |p ^n| \wedge \\
&\wedge |\2 -\2^{(n)}(\r)| \leq |p ^n| \big\}.
\end{split}
\end{equation*}
Clearly $(\br,\3,\4) \in \cW_n(\overline \alpha)$.
 \item[(ii)] $\cW^1_n(\overline \alpha) $ is the $K$--rational domain
\begin{eqnarray*}
 \big\{(\r,\1) : |p| \leq |\r^{n!}| \leq 1 \wedge \bigwedge_{i=1}^\ell |\r - \alpha_i| = 1
\wedge |\1 -\1^{(n)}(\r)| \leq |p ^n| \big\},
\end{eqnarray*}
and $\cW^2_n(\overline \alpha) $ is the $K$--rational domain
\begin{eqnarray*}
 \big\{(\r,\2) : |p| \leq |\r^{n!}| \leq 1 \wedge \bigwedge_{i=1}^\ell |\r - \alpha_i| = 1
\wedge |\2 -\2^{(n)}(\r)| \leq |p ^n| \big\}.
\end{eqnarray*}
 \item[(iii)] Let $n = \max(n_1,n_2)$.  Then $\cW_{n_1,n_2}(\overline \alpha) $ is the $K$--rational domain
\begin{equation*}
\begin{split}
 \big\{(\r,\1,\2) : |p| \leq |\r^{n!}| \leq 1 \wedge \bigwedge_{i=1}^\ell |\r - \alpha_i| = 1
&\wedge \\
\wedge |\1 -\1^{(n_1)}&(\r)| \leq |p ^{n_1}| \wedge  |\2 -\2^{(n_2)}(\r)| \leq |p ^{n_2}| \big\}.
\end{split}
\end{equation*}
\item[(iv)]  For $\cU$ a $K$--rational domain we denote by $\cO_\cU$ the ring of (strictly convergent) analytic functions on $\cU$ (see \cite{BGR} 7.3.2).
\end{defnn}

\begin{remm}\label{rem2} \item[(i)] $\cO_\cU$ depends only on the set $\cU$ and not on a particular definition of $\cU$.  Observe that if $n_1 \leq n_2$ then $\cW_{n_1,n_2}(\overline \alpha) =$
\begin{equation*}
\begin{split}
 \big\{(\r,\1,\2) : |p| \leq |\r^{n!}| \leq 1 \wedge \bigwedge_{i=1}^\ell |\r - \alpha_i| = 1& \wedge \\
\wedge |\1 -\1^{(n_2)}&(\r)| \leq |p ^{n_1}|
\wedge  |\2 -\2^{(n_2)}(\r)| \leq |p ^{n_2}| \big\}
\end{split}
\end{equation*}
(here we have replaced $ |\1 -\1^{(n_1)}(\r)| \leq |p^{n_1}| $ by $ |\1 -\1^{(n_2)}(\r)| \leq |p^{n_1}| $ in the definition. Recall that $n = \max(n_1,n_2)$).
\item[(ii)]  We are being a bit sloppy with the definitions of the $\cW_{*}(\overline \alpha)$ since we are allowing the $\alpha_i \in \Kalg$  However, if the set $\{\alpha_1, \dots , \alpha_\ell \}$ is taken to be closed under conjugation (which it always can be), then $\cW_{*}(\overline \alpha)$ can always be defined over $K$ (see \cite{CLR1} section 3 or \cite{CL1} section 5).
\end{remm}

\begin{lemm}
Let $\cU$ be a $K$--rational domain with $(\br, \3, \4) \in \cU$.  Then there is an integer $n \in \bN$ and
$\overline \alpha = (\alpha_1, \dots , \alpha_\ell) \in ( \Kalgo \setminus \Kalgoo)^\ell$ such that  $(\br,\3,\4) \in \cW_n(\overline \alpha) \subset \cU$.

\end{lemm}
\begin{proof}
$\cU$ can be defined by a finite set of inequalities of the form
\begin{eqnarray}
 \epsilon \leq |g_i(\r,\1,\2)| \leq h_i(\r,\1,\2)|.
\end{eqnarray}
Here $0 < \epsilon \in \bR$ and the $g_i, h_i \in \Ko[\r,\1,\2]$.  Restricting to $\cU \cap \cW_n(\emptyset)$ for $n$ such that
$|p ^n| < \epsilon^2$ these inequalities reduce to
\begin{eqnarray}\label{5.5}
 \epsilon \leq |g_i(\r,\1^{(n)}(\r),\2^{(n)}(\r))| \leq |h_i(\r,\1^{(n)}(\r),\2^{(n)}(\r))|.
\end{eqnarray}
Since these inequalities are satisfied by $\br$ we have
\begin{eqnarray*}
 \|(\r^{n!})^k g_i(\r,\1^{(n)}(\r),\2^{(n)}(\r))\| \leq \|(\r^{n!})^k h_i(\r,\1^{(n)}(\r),\2^{(n)}(\r))\|
\end{eqnarray*}
where $k \in \bN$ is $\geq$ the degrees of the $g_i$ and $h_i$, and $\| \cdot \|$ denotes the gauss-norm.
But then this inequality is also satisfied for all $\r$ with  $|\r| = 1$ except for $\r$ in a finite union of open discs of radius $1$ all defined over $\Kalg$ (or even $K$).  Since the inequalities \ref{5.5} are satisfied by
$\br$, they are satisfied for all $\r$ with $1 - \delta \leq |\r| < 1$ for some $0 < \delta \in \bR$.  Increasing $n$ so that also $|p ^n| < \delta$ and using $\overline \alpha$ to exclude the open discs gives us the required
$\cW_n(\overline \alpha)$.
\end{proof}

\begin{lemm}\label{lem2}
Let $F(\r,\1,\2) \in \cO_{\cW_n(\overline \alpha)}$ (the ring of analytic functions on the rational domain
$\cW_n(\overline \alpha)$), and assume that
\begin{eqnarray*}
X(\br)|_{\cW_n} \subset \big\{(\r,\1,\2) \in \cW_n : F(\r,\1,\2) = 0 \big\} =: V_F.
\end{eqnarray*}
Then, if we define
\begin{eqnarray*}
\overline X(\br) :=  \big\{(\br,\1,\2) : \1 + \2 z_2(\br) = f(z_2(\br) \big\} \subset (K_1^\circ)^3,
\end{eqnarray*}
we also have that $\overline X(\br)|_{\cW_n} \subset V_F$.
\end{lemm}
\begin{proof}
Since $F \in \cO_{\cW_n(\overline \alpha)}$ and $X(\br)|_{\cW_n((\overline \alpha)} \subset V_F$, we have
for $(\br,\1,\2) \in X(\br)|_{\cW_n(\overline \alpha)}$ that $F(\br,\1,\2) = 0$.  Now for $(\r,\1,\2) \in X$ we have
$(\1 + \zeta\2 - f(\zeta))|_{\zeta = z_1(\r)} = 0$, and on $\cW_n(\overline \alpha)$ we have for $i=1,2$ that
$\xi_i = \xi_i^{(n)}(\r) + p^n\eta_i$ where $\eta_1, \eta_2$ are new variables that vary over $\Kalgo$
(or $K_1^\circ$).

Let $\cW^+_n(\overline \alpha) := \{(\1,\2,\r,\zeta): (\1,\2,\r)\in \cW_n(\overline \alpha)
\wedge |\r \zeta^2 - \zeta + 1|\leq|p^n \rho^{n!}\}. $  On $\cW^+_n(\overline \alpha)$ we have that
$\xi_i = \xi_i^n(\r) + p^n \eta_i$ for $i=1,2$, so we may consider the variables to be $\eta_1, \eta_2, \zeta, \r,
\frac{p}{\r^{n!}}, \frac{1}{\r - \alpha_i} \text{ and } \frac{\r\zeta^2 -\zeta +1}{p^n \r^{n!}}.$  By direct computation, using that the $\xi_i^{(n)}$ are approximations to the Weierstrass data of $f(\zeta)$ on division by
$\zeta^2 - \frac{1}{\r}\zeta + 1$, we see that $\1 + \zeta \2 - f(\zeta)$ is regular in $\eta_1$ of degree $1$.
Hence, on $\cW^+_n(\overline \alpha)$ we have that
\begin{equation*}\label{5.66}
F(\r,\1,\2) = G(\r,\2, \zeta) + (\1 + \zeta \2 - f(\zeta)) Q
\end{equation*}
where $G$ and $Q \in \cO_{\cW^+_n(\overline \alpha)}$.

Now, by (formal) Weierstrass Division of $G$ by $\zeta^2 - \frac{1}{\r}\zeta + \frac{1}{\r}$ we have
\begin{equation}\label{5.7}
F(\r,\1,\2) = H_1(\r,\eta_2) + H_2(\r,\eta_2)\zeta + (\zeta^2 - \tfrac{1}{\r} \zeta +1) Q_1 + (\1 + \zeta \2 - f(\zeta)) Q
\end{equation}

Here $H_1$, $H_2$ and $Q_1$ are  {\it not} strictly convergent on the domain $\cW^+_n(\overline \alpha) $.  However $H_1, H_2$ and $Q_1$ are strictly convergent in $\eta_2$ with coefficients functions in
$\Ko[\r,\frac{1}{\r},\frac{1}{\r-\alpha_1}, \dots , \frac{1}{\r-\alpha_\ell}, \zeta, \frac{\r \zeta^2 - \zeta + 1}{p^n\r^{n!}}]~\widehat{} \quad =: A^+$ where \quad $\widehat{}$ \quad denotes completion in the norm on $\Ko$,  i.e. we can write
\begin{eqnarray*}
H_1(\r,\eta_2) &=& \sum_{j=0}^\infty h_{1j}(\r)\eta_2^j\\
H_2(\r,\eta_2) &=& \sum_{j=0}^\infty h_{2j}(\r)\eta_2^j
\end{eqnarray*}
where the $h_{ij} \in \Ko[\r,\frac{1}{\r},\frac{1}{\r-\alpha_1}, \dots , \frac{1}{\r-\alpha_\ell}]~\widehat{}$ \quad and
$\|h_{ij}\| \to 0$ as $j \to \infty$. (Here $\|\cdot \|$ is the gauss-norm on
$\Ko[\r,\frac{1}{\r},\frac{1}{\r-\alpha_1}, \dots , \frac{1}{\r-\alpha_\ell}]~\widehat{}$ ~ ).
Because $K_1$ is complete, and contains no infinitesimals, all the power series in $A^+$ define functions
on\\ $\cW^+_n(\overline \alpha) \cap \{(\1,\2,\r,\zeta): 1-\frac{1}{k} \leq |\r| < 1~ \forall k \in \bN \wedge |\r\zeta^2 - \zeta +1| \leq |p^n\r^{n!}\}$.  Indeed, they define functions on $\cW^{++}_n(\overline \alpha)$ defined like $\cW^+_n(\overline \alpha)$ but allowing $\zeta$ to vary over the slightly bigger disc $\{\zeta \in K_1 : |\zeta| \leq 1 + \frac{1}{k} ~\forall k \in \bN \}$, i.e. on
\begin{equation*}
\begin{split}
\cW^{++}_n(\overline \alpha) :=&  \big\{(\1,\2,\r,\zeta) \in (K_1)^3 \times K_1): (\1,\2,\r)\in \cW_n(\overline \alpha) \wedge \\
\wedge & |\r \zeta^2 - \zeta + 1|\leq|p^n \rho^{n!}|
  \wedge \big(1-\tfrac{1}{k} \leq |\r| < 1  \wedge |\zeta| \leq 1 + \tfrac{1}{k} ~\forall k \in \bN \big) \big\}.
\end{split}
\end{equation*}

 Note that also
$h_{ij}(\br) \in K_1$ and $h_{ij}(\br) \to 0$ as $j \to \infty$.
Hence substituting $z_1(\br)$ for $\zeta$ and $\br$ for $\r$ in equation (\ref{5.7}) we have that
\begin{eqnarray*}
H_1(\br,\eta_2) + H_2(\br,\eta_2)z_1(\br) = 0
\end{eqnarray*}
for all $\eta_2 \in K_1^\circ$, i.e.
\begin{eqnarray*}
\sum_{j=0}^\infty [h_{1j}(\br) + h_{2j}(\br)z_1(\br)]\eta_2^j = 0
\end{eqnarray*}
for all $\eta_2 \in K_1^\circ$.  Hence
\begin{eqnarray*}
h_{1j}(\br) + h_{2j}(\br)z_1(\br) = 0
\end{eqnarray*}
for all $j$.

Next we will show that $h_{ij}(\br) = 0$ for $i=1,2$ and $j= 0,1,2, \cdots$.   Suppose not.  Then we have for some $j$ that  $h_{1j}(\br) + h_{2j}(\br)z_1(\br) = 0$ but either $h_{1j}(\br) \neq 0$ or  $ h_{2j}(\br) \neq 0$.  Since
$|z_1(\br)| =1$ we must have both $h_{1j}(\br) \neq 0$ and  $ h_{2j}(\br) \neq 0$.  Hence
$|h_{1j}(\br)| = |h_{2j}(\br)|$.  From the definition of $\br$ we have that $\|h_{ij}(\r)\| - |h_{ij}(\br)|$ is infinitesimal.  Hence,
multiplying by a suitable element of $K \setminus\{0\}$ we may assume  that $\|h_{1j}(\r)\| = (\|h_{2j}(\r)\|=1$.  Let \quad $\widetilde{}$\quad denote on
$\Ko$ the residue modulo $\Koo$ and on $K_1^\circ$ the residue modulo
$K_1^\circ \cdot \Koo = \{x \in K_1^\circ: |x| \leq |y| \text{ for some } y \in \Koo\}$.  Then
\begin{eqnarray*}
\widetilde h_{1j}(\widetilde \br) + \widetilde h_{2j}(\widetilde \br) \widetilde{z_1(\br)} = 0.
\end{eqnarray*}
Since $\widetilde h_{1j}(\r)$ and $\widetilde h_{2j}(\r)$ are rational functions of $\r$, we  have that
$\frac{\widetilde h_{1j}(\widetilde \br)}{\widetilde h_{2j}(\widetilde \br)}$ is a solution of the equation $\widetilde \br \zeta^2 - \zeta +1 = 0$ in $\widetilde K (\widetilde \br)$.  But this equation has no rational solutions ( $\widetilde \br$ is transcendental over $\widetilde K$.  Look at the zeros and poles of the rational function).  Since $\br$ is transcendental over $K$, also $h_{ij}(\r) = 0$ for $i=1,2$ and $j= 0,1,2, \cdots$.  Thus
\begin{equation}\label{5.8}
F(\r,\1,\2) =  (\zeta^2 - \tfrac{1}{\r} \zeta +1) Q_1 + (\1 + \zeta \2 - f(\zeta)) Q.
\end{equation}
This is an equation among elements of $A^+$, and hence an equation among the functions they define on
$\cW^{++}_n(\overline \alpha)$.
Hence from equation  \ref{5.8}, for any $(\br,\1,\2) \in \overline X(\br)|_{\cW_n(\overline \alpha)}$ we have
\begin{equation} \label{fns}
\begin{split}
F(\br,\1,\2) = +(z_2^2(\br) - \tfrac{1}{\br}z_2(\br) + \tfrac{1}{\br})Q_1(\br,&\eta_2,z_2(\br)) \\
+  ((\1 + z_2(\br)\2 - f(z_2&(\br))Q(\br,\1,\2,z_2(\br)),
 \end{split}
\end{equation}
i.e. $F(\br,\1,\2) = 0$.
Again, this equation is to be understood as an equation among functions on $\cW^{++}_n(\overline \alpha)$.
(Since $|z_2(\br)| > 1$,
$f(z_2(\br), Q(\br,\1,\2,z_2(\br))$ and
$Q_1(\br,\eta_2,z_2(\br))$ are not
evaluated using the analytic structure, but by using the completeness of $K_1$.  As we observed above,  power series in $A^+$ define functions on $\cW^{++}_n(\overline \alpha)$.
Indeed,  for every $m$, modulo $p^{m+1}$ all the functions $F, f, Q, Q_1$ are congruent to polynomials in the various variables with coefficients rational functions of $\br$.  Hence $|F(\br,\1,\2)| < |p^m|$ for every $m \in \bN$.
Since there are no infinitesimals in $K_1$ we have $F(\br,\1,\2) = 0$).  This completes the proof of the lemma.
\end{proof}

\begin{remm}Here is an alternate proof of Lemma \ref{lem2} that may be of some interest.
$\bR \subset \bR^*$ is the nonprincipal ultraproduct of $\bR$.
Let $\bR^*_f := \{x \in \bR^* : \exists y_{ \in \bR}(|x| \leq |y|\}$ be the finite part of $\bR^*$, and
$I_{inf} := \{x \in \bR^* : \forall n_{ \in \bN}( |y| < \frac{1}{n})  \}$ be the ideal of infinitesimals in  $\bR^*_f$.
Let $\bR_1 = \bR^*_f / I_{inf}$.  Then $K_1$ is a normed field with norm $|\cdot |_1 \to \bR_1$ and $K_1$ is complete in $|\cdot |_1$, i.e. if $|a_i|_1 \to 0$ then $\sum_i a_i \in K_1$.  Consider the norm
$|\cdot |_2 : K_1 \to \bR$ defined by: $|x|_2 :=$ the standard part of $|x|_1$.  Denote $K_1$ with norm
$|\cdot |_2$ by $K_2$.  Observe that for $a_i \in K_1$ we have $|a_i|_1 \to 0$ if and only if  $|a_i|_2 \to 0$ so $K_2$ is also complete.  Let $K_2^\circ := \{x \in K_2 : |x|_2 \leq 1\} =  \{x \in K_1 : \forall n_{ \in \bN}(|x|_1 < 1+ \frac{1}{n})  \}$.  Then $K_2^\circ \supsetneq \Ko_1 = \{x \in K_1 : |x|_1 \leq 1\}$.  Both $K_1$ and $K_2$ are algebraically closed valued fields with analytic $\cA$--structure, say via $\sigma_1$ and $\sigma_2$ respectively.  If $g \in A_m$ and $x_1,\cdots,x_m \in \Ko_1$ then
$\sigma_1(g)(x_1,\cdots,x_m) = \sigma_2(g)(x_1,\cdots,x_m)$ and $|\sigma_2(g)(x_1,\cdots,x_m)|_2 = $ the standard part of $|\sigma_1(g)(x_1,\cdots,x_m)|_1$.  However, if at least one of the $x_i \in \Ko_2 \setminus \Ko_1$ then $\sigma_1(g)(x_1,\cdots,x_m) = 0$.  Consider the lines
\begin{eqnarray*}
X(\br) &:=& \{(\br,\1,\2):  \1 + z_1(\br) \2 = f(z_1(\br))\} \text{ and }\\
\overline X(\br) &:=& \{(\br,\1,\2):  \1 + z_2(\br) \2 = f(z_2(\br))\}.
\end{eqnarray*}
There is a $|\cdot|_2$--preserving $K$--automorphism $\pi$ of $K_2$ that leaves $\br$ fixed but interchanges $z_1(\br)$ and $z_2(\br)$. Since $|\pi(x)|_2 = |x|_2$ we have $||\pi(x)|_1 - |x|_1| \in I_{inf}$.
Suppose $F(\r,\1,\2) \in \cO_{\cW_n(\overline \alpha)}$ is such that
$\sigma_1(F)(\br,\1,\2)$ vanishes on $X(\br)|_{\cW_n(\overline \alpha)}$.  Then also
$\sigma_2(F)(\br,\1,\2)$ vanishes on $X(\br)|_{\cW_n(\overline \alpha)}$ and hence
$\sigma_2(F)(\br,\1,\2)$ vanishes on
$\pi(X(\br)|_{\cW_n(\overline \alpha)}) = \overline X(\br)|_{\cW_n(\overline \alpha)}$.  Now for
$(\br,\1,\2) \in \cW_n(\overline \alpha)$ we have $|\xi_i^{(n)}(\br)|_1 = |\frac{p}{\br}| \leq |p|^{\frac{1}{2}}$ and
$|\xi_i - \xi_i^{(n)}(\br)|_1 \leq 1$.  Hence also $|\pi(\xi_i)|_1 \leq 1$ so we have that
$\sigma_1(F)(\br,\1,\2)$ vanishes on $\overline X(\br)|_{\cW_n(\overline \alpha)}$.
\end{remm}

\begin{lemm}\label{lem3}
Let $F(\r,\1,\2) \in \cO_{\cW_n(\overline \alpha)}$ and suppose that $F(\br,\3,\4) = 0$ and
$\frac{\partial F}{\partial \xi_1}(\br,\3,\4) \neq 0$.  Then there are  $n', n'' \in \bN$ and an $\overline \alpha'$ such that
$F(\r,\1,\2)$ as an element of $\cO_{\cW_{n,n''}(\overline \alpha')}$ is regular of degree $1$ in $\1$, after division by a suitable unit.  Hence there is a $\widehat \xi_1 = \widehat \xi_1(\r,\2) \in \cO_{\cW^2_{n''}(\overline \alpha')}$ such that
$F(\r, \widehat \xi_1(\r,\2),\2) = 0$ as an element of $ \cO_{\cW_{n''}(\overline \alpha')}$.  The same holds with the roles of $\1$ and $\2$ interchanged.
\end{lemm}
\begin{proof}
For strictly convergent $g$, let $g^{(n')}$ be a polynomial that is congruent to $g$ modulo $p ^{n'}$, i.e. the truncation of $g$ at size
$p ^{n'}$.
Since $\frac{\partial F}{\partial \xi_1}(\br,\3,\4) \neq 0$ we have that
$|\frac{\partial F}{\partial \xi_1}(\br,\3,\4)|^2 >  |p^{n'}|$ for some $n'$.  Then
$|\frac{\partial F^{(n')}}{\partial \xi_1}\big(\br,\3^{(n')},\4^{(n')}\big)|^2 >  |p^{n'}|$.  Observe that
\begin{eqnarray*}
\frac{\partial F^{(n')}}{\partial \xi_1}\big(\br,\3^{(n')},\4^{(n')}\big) \in \frac{1}{\br^{n''!}}\Ko[\br]
\end{eqnarray*}
for some $n''$.
Hence for suitable $n'' \geq n'$ and $ \overline \alpha'$
\begin{eqnarray*}
\big|\frac{\partial F^{(n')}}{\partial \xi_1}(\r,\1,\2)\big|^2 \geq  |p^{n'}|
\end{eqnarray*}
for all $(\r,\1,\2) \in \cW_{n''}(\overline \alpha')$.  On $\cW_{n',n''}(\overline \alpha')$ we can write
$\1 = \1^{(n'')}(\r) + p^{n'}\eta_1$ and $\2 = \2^{(n'')}(\r) + p^{n''}\eta_2$, where the $\eta_i$ are new variables.
Indeed, on the set $\{\r: |p| \leq |\r^{n''!}| \leq 1\}$ we have $|\1^{(n')}(\r) - \1^{(n'')}(\r)| < |p^{n'}|$, so  taking
$\1^{(n'')}(\r)$ as the ``center'' of the disc $|\1 - \1^{(n')}| \leq |p^{n'}|$ instead of $\1^{n'}(\r)$ does not change the rational domain or its ring of functions (Remark \ref{rem2}).
Writing $F(\r,\1(\r,\eta_1),\2(\r,\eta_2))$ as a power series in $\eta_1$ and dividing by the unit
$$
p^{n'}\frac{\partial F}{\partial \xi_1}\big(\r,\1^{(n'')},\2^{(n'')} + p^{n''}\eta_2\big),
$$
which is the coefficient of $\eta_1$,
makes $F$ regular of degree $1$ in $\eta_1$.  The existence of $\widehat \xi_1(\r, \2)$ follows by Weierstrass preparation.
\end{proof}

We now proceed with the proof that $X$ is not $\cL_\cA^D$--semianalytic on any $\cW_n(\overline \alpha)$.
\bigskip
\begin{case} We will show first that $X$ is not $\cL_\cA$--semianalytic on any $\cW_n(\overline \alpha)$.  \end{case}
\noindent Suppose it were.  Then there would be an $F \in \cO_{\cW_n(\overline \alpha)}$ such that
$X(\br)|_{\cW_n(\overline \alpha)} \subset V_F$.
Since $X(\br)|_{\cW_n(\overline \alpha)}$ is a line, we may assume that $F$ is irreducible in
$ \cO_{\cW_n(\overline \alpha)}$ and in  $ \cO_{\cW_n'(\overline \alpha')}$ for all $n' \ge n$ and all $\overline \alpha'$.  By Lemma \ref{lem2}, also
$\overline X(\br)|_{\cW_n(\overline \alpha)} \subset V_F$.
Hence the curve $V_F(\br)$ in $\cW_n$ defined by $F(\br,\1,\2)=0$ would be singular at $(\br,\3,\4)$.  Increasing $n$ if necessary, we may assume that $(\br,\3,\4)$ is the only singular point on
the curve $F(\br,\1,\2) = 0$ in $\cW_n(\overline \alpha)$.  But then, since a point is a smooth $0$--dimensional variety, by  \cite{LR6} or \cite{Ce1}, for some
$n' \geq n$ and some $\overline \alpha ' $ we would have that $\3$ and $\4$  are
$\cO_{\cW_n(\overline \alpha ' )} \text{--analytic}$ functions of $\br$.  But this is not the case by Remark \ref{rem}.
Alternatively we can see this as follows:  For some suitable derivative $G$ of $F$ the curve $V_G(\br,\1,\2)$ contains the point $(\br,\3,\4)$ and is smooth at that point. If $V_G \not\subset V_F$,  by Lemma \ref{lem3} we can solve $G= 0$ for $\1$, say, as an analytic function $\widehat \xi_1(\r,\2)$ of $\r$ and $\2$, and substitute this function into $F$ to get $F(\r, \widehat \xi_1(\r,\2),\2) \not\equiv 0$. Use the equation
$F(\r, \widehat \xi_1(\r,\2),\2) = 0$ and Lemma \ref{lem3} to solve for $\widehat \2$ as an analytic function of $\r$,  contradicting Remark \ref{rem}.  If $V_G  \subset V_F$, since $V_G(\br)$ is smooth, we know that
$X(\br)|_{\cW_n(\overline \alpha)} \not\subset V_G$ and hence that $V_F \neq V_G$.  Then, since $G$ is (strictly convergent) regular in $\1$ of degree $1$, by (strictly convergent) Weierstrass division,
$G$ would divide $F$, contradicting the fact that $F$ is irreducible.
\bigskip
\begin{case} Finally we show that $X$ is not $\cL_\cA^D$--semianalytic.  \end{case}
\noindent Suppose it were.  Choose an $\cL_\cA^D$--semianalytic representation on a
 $\cW_n(\overline \alpha)$ of lowest possible complexity  --  i.e. using the smallest number of  non-analytic $D$--functions (necessarily $\geq 1$ by Case 1) and with innermost $D$--functions reduced (i.e. with numerator and denominator having no common factor).  Let $D(g,h)$ be an   innermost $D$--function that is not analytic.  Then $g,h \in \cO_{\cW_n(\overline \alpha)}$ and   necessarily $g(\br,\3,\4) =  h(\br,\3,\4) = 0$, and $g$ and $h$ have no common factor in any   $ \cO_{\cW_n'(\overline \alpha')}$.
Considering the two curves $\cC_g := \{(\br,\1,\2) : g(\br,\1,\2) = 0 \}$ and  $\cC_h := \{(\br,\1,\2) : h(\br,\1,\2) = 0 \}$, we now proceed as in Case 1 to show that $\3, \4$ are analytic functions of $\br$, again contradicting Remark \ref{rem}.
This completes the proof of Theorem \ref{cex}.

\begin{remm}
The set $X$ is clearly both $\cL_{sep}$--semianalytic and $\cL_\cA^h$--semianalytic.  It is also $\cL_\cA$--semianalytic in a neighborhood of every point of $\Kalg$. Indeed it is $\cL_\cA$--semianalytic on the sets described by $|\r| = 1$ and $|\r| \leq 1 - \delta$, for every $\delta < 1, \delta \in \bR$.
\end{remm}

\appendix

\section{Some comments on analytic structures from \cite{CL1}}

In this appendix we recall some definitions from \cite{CL1}, sometimes in a slightly modified form, we give some complements and small amendments to \cite{CL1}, notably in Definition \ref{quotes}, Lemma \ref{cor}, and Remarks \ref{correction0} \ref{correction}, \ref{correction2}. We also give some supplementary discussion of various topics -- other strong Noetherian Properties, extension of parameters, and Weierstrass systems with side conditions.

\subsection{}

Let us first recall the notion of Separated pre-Weierstrass system of \cite{CL1}, (`separated' means that the power series involve two kinds of variables, see Remark \ref{rem:sep}).

 \begin{defn}[Separated pre-Weierstrass system, cf. \cite{CL1} Definitions 4.1.2, 4.1.3 and 4.1.5]  \label{sepWS} Let $B$ be a ring and $I \subset B$ a proper ideal. Let $B_{m,n} \subset B[[\xi,\rho]]$ be $B$-algebras. We call $\cB := \{B_{m,n}\}_{m,n \in \bN}$ a {\it $(B,I)$-separated \textbf{pre}-Weierstrass system} if it satisfies the following conditions (1) - (6) for all $m,n,m',n'$.   (Let $m \leq m'$ and $n \leq n'$ be natural numbers, and $\xi=(\xi_1,\ldots,\xi_m)$, $\xi'=(\xi_1,\ldots,\xi_{m'})$, $\xi''=(\xi_{m+1},\ldots,\xi_{m'})$, $\rho=(\rho_1,\ldots,\rho_{n})$,
   $\rho'=(\rho_1,\ldots,\rho_{n'})$, and
 $\rho''=(\rho_{n+1},\ldots,\rho_{n'})$ be variables).

\item[(1)]$ B_{0,0} = B,$\
 \item[(2)]
$B_{m,n}\subset B[[\xi,\rho]]$ and is closed under any permutation of the $\xi_i$ or of the $\rho_j$,\
 \item[(3)]
$ B_{m,n} [\xi'', \rho''] \subset B_{m',n'},$\
 \item[(4)]the image $(B_{m,n})\;\widetilde{}$
 of $B_{m,n}$ under the residue map $\, \widetilde{}:B[[\xi,\rho]] \to \widetilde {B}[[\xi,\rho]]$ is a subring of $
 \widetilde{B}[\xi][[\rho]]$, and\
\item[(5)] if
 $f \in B_{m',n'}$, say $f =
\sum_{\mu\nu}\overline{f}_{\mu\nu}(\xi,\rho)(\xi'')^\mu(\rho'')^\nu$,
then the $\overline{f}_{\mu\nu}$ are in $B_{m,n},$\

\item[(6)] The two usual Weierstrass Division Theorems hold in the $B_{m,n}$,
namely, for $f,g\in B_{m,n}$:

{\em (a)}  If $f$ is regular in $\xi_m$ of degree $d$ (Definition \ref{regular}(i)), then there
exist uniquely determined elements $q\in B_{m,n}$ and $r\in
B_{m-1,n}[\xi_m]$ of degree at most $d-1$ such that $g=qf+r$.

{\em (b)} If $f$ is regular in $\rho_n$ of degree $d$ (Definition \ref{regular}(ii)), then there
exist uniquely determined elements $q\in B_{m,n}$ and $r\in
B_{m,n-1}[\rho_n]$ of degree at most $d-1$ such that $g=qf+r$.

\end{defn}

Since $(A, I)$ is quite general there may be pairs $(a,b)$ in $A$ whose relative size is not determined by $(A,I)$.  In each field with analytic $\cA$--structure, say via $\sigma$,  the relative size of $\sigma(a)$ and $\sigma(b)$ is determined.  We handle this situation by means of Definition \ref{frac} (and Definition \ref{con7}).  In the case that $A = \Ko$ and $I = \Koo$ with $K$ a valued field this complication does not arise.  In the case of Theorem \ref{ex3}  we can also avoid this complication by using the stronger property (7$'$) of Definition \ref{SNPstr'} in place of Definition \ref{con7}.

We recall the following definitions from \cite{CL1}.

\begin{defn}[Rings of fractions]\label{frac}
Let $\cA=\{A_{m,n}\}_{m,n \in \bN}$ be a separated
 $(A,I)$-pre-Weierstrass system. Inductively define the concept that an $A$-algebra $C$ is a \emph{ring of
$\cA$-fractions with proper ideal $C^\circ$} and with \emph{rings
$C_{m,n}$ of separated power series ``over" $C$} by

\item[(i)] The ring $A$ is
a ring of $\cA$-fractions with ideal $A^\circ=I$ and with rings of
separated power series the $A_{m,n}$ from the system $\cA$.

\item[(ii)]
If $B$ is a  ring of $\cA$-fractions and $d$ in $B$ satisfies
$C^\circ \not=C$, with
$$C:= B/dB,$$
$$C^\circ := B^\circ/dB,$$
 then $C$ is a ring of $\cA$-fractions with proper ideal $C^\circ$ and
$C_{m,n}:= B_{m,n}/dB_{m,n}$.

\item[(iii)] If $B$ is a ring of $\cA$-fractions and $c,d$ in $B$
satisfy $C^\circ\not = C$ with
 $$C=B\langle \frac{c}{d} \rangle := B{}_{1,0}/(d\xi_1 -
c),$$
 $$C^\circ:=(B^\circ)B\langle \frac{c}{d}\rangle$$
 then $C$ is a ring of $\cA$-fractions with proper ideal
$C^\circ$ and $C_{m,n}:=
 B_{m+1,n}/(d\xi_1 - c)$.

\item[(iv)] If $B$ is a ring of $\cA$-fractions and $c,d$ in $B$ satisfy $C^\circ\not = C$, with $$C= B[[\frac{c}{d}]]_s:=B_{0,1}/(d\rho_1 -
c),$$ $$C^\circ:=(B^\circ, \rho_1) B_{0,1}/(d\rho_1 - c),$$ and
$(B^\circ, \rho_1)$ the ideal generated by $B^\circ$ and $\rho_1$,
then $C$ is a ring of $\cA$-fractions with proper ideal $C^\circ$
and $C_{m,n}:=B_{m,n+1}/(d\rho_1 - c)$.
 \end{defn}

\begin{defn}\label{quotes}
In  Definition \ref{frac} we put the word {\it ``over''} in quotes because $C_{m,n}$ is not in general a ring of power series with coefficients from $C$.  However, if $f \in C_{m,n}$, the Taylor coefficients of $f$ are well defined as elements of $C$ by repeated Weierstrass division.  In other words, for each $(\mu,\nu)$ there is a unique $\overline{c}_{\mu,\nu} \in C$ such that for $\mu < \mu'$ and $\nu < \nu'$, $\overline{c}_{\mu,\nu}$ is the coefficient of $\xi^\mu\rho^\nu$ in $f \mod(\xi^{\mu'}, \rho^{\nu'})$.  We can define the {\it series of $f$} by $\cS(f) := \sum \overline{c}_{\mu,\nu}\xi^\mu\rho^\nu \in C[[\xi,\rho]]$.

In all cases of Definition \ref{frac} define $C^\circ_{m,n} := \{f \in C_{m,n} : \text{ for all $\mu$, }  \overline{c}_{\mu,0} \in C^\circ\}$

\end{defn}

\begin{defn}\label{con7}We call a separated pre-Weierstrass system $\cB$ (Definition \ref{sepWS}) a {\it separated Weierstrass system} if it satisfies

\item[(7)]For every ring of $\cB$--fractions $C$ the following condition is satisfied.
If $f \in C_{m,n}$ and $\cS(f) = \sum_{\mu,\nu}\overline{c}_{\mu\nu}\xi^\mu\rho^\nu$, then there is a finite set
 $J \subset \bN^{m+n}$ and for each $(\mu,\nu) \in J$ there is a
 $g_{\mu\nu} \in C_{m,n}^\circ$ such that
 $$
 f = \sum_{(\mu,\nu) \in J} \overline{c}_{\mu\nu}\xi^\mu\rho^\nu(1+g_{\mu\nu}).
 $$
\end{defn}

\begin{lem}\label{cor} If $\cB = B_{m,n}$ is a separated Weierstrass system, then for every ring of $\cB$--fractions $C$,
$C_{m,n}$ is indeed a ring of power series over $C$.
\end{lem}
\begin{proof}Property (7) above guarantees that the homomorphism
$$
f \mapsto \cS(f) = \sum \overline{c}_{\mu,\nu}\xi^\mu\rho^\nu
$$
is an isomorphism.
\end{proof}

In the light of this Lemma we will henceforth identify $C_{m,n}$ with $\cS(C_{m,n})$.

The treatment in \cite{CL1} implicitly assumed that $C_{m,n}$ is a power series ring over $C$.  The treatment here is better in that it omits that assumption, and the statement of the Strong Noetherian Property uses $\cS(f)$ to define the $\overline c_{\mu,\nu}$. (Definition 4.1.5(c) of \cite{CL1} should be replaced by Definition \ref{con7} above).
Once we have assumed property (7), the distinction between $C_{m,n}$ and $\cS(C_{m,n})$ disappears and $C_{m,n}$ is indeed a power series ring, as observed above.

Restating (7) in our context, we have:

\begin{lem}[Criterion]\label{SNPstr}A good strictly convergent pre-Weierstrass system $\cA = \{A_m\}$ is a strictly convergent Weierstrass system if (and only if) $\cA^H$ satisfies condition (7) of Definition \ref{con7}, i.e. if
\item[(vii)]For every ring of $\cA^H$--fractions $C$ the following is true.
If $f \in C_{m,n}$ and $\cS(f) = \sum_{\mu,\nu}\overline{c}_{\mu\nu}\xi^\mu\rho^\nu$, then there is a finite set
 $J \subset \bN^{m+n}$ and for each $(\mu,\nu) \in J$ there is a
 $g_{\mu\nu} \in C_{m,n}^\circ$ such that
 $$
 f = \sum_{(\mu,\nu) \in J} \overline{c}_{\mu\nu}\xi^\mu\rho^\nu(1+g_{\mu\nu}).
 $$
\end{lem}

\begin{rem}
It does not follow that a field $F$ with strictly convergent analytic $\cA$-structure is necessarily henselian.  See the example in \cite{CL1}, Remark 4.5.13.  This is why we must require that the fields with analytic $\cA$--structure be henselian.
\end{rem}

\begin{rem}\label{correction0}
We amend Definition 4.5.14 and Theorem 4.5.15 of \cite{CL1}.  In Definition 4.5.14 of \cite{CL1}, the order relation $<$ on the value group sort should be part of the language $\cL_\cA$, in order for the quantifier elimination result of Theorem 4.5.15 to make sense.
\end{rem}

\begin{rem}\label{correction}
The alternative proof of quantifier elimination, based on compactness, that is given in section 6 of \cite{CL1} (in the proof of Theorem 6.3.7) is not well-presented.  In order to obtain a proof along these lines the results used (\cite{CLR1}  Lemma 3.16 and \cite{CL1} Theorem 5.5.3) should be made more uniform, allowing variables as well as constants to occur in the terms.   Proposition \ref{553amend} is the needed uniform version of  \cite{CL1} Theorem 5.5.3.
\end{rem}

\begin{prop}\label{553amend}
Let $\cA$ be a separated Weierstrass system and let $K$ be a valued field with separated analytic $\cA$-structure.  Let $\cL_{\cA(K)}$ be the language of valued fields, $\langle 0,
1, +, \cdot, (\cdot)^{-1}, |\cdot| \rangle$, augmented with function
symbols for all the elements of $\bigcup_{m,n}A_{m,n}(K)$. (We
extend functions $f \in A_{m,n}(K)$ by zero outside
$(\Kalgo)^m\times(\Kalgoo)^n$.) Let $x$ be several variables and $y$ one variable, and let $\tau(x,y)$
be a term of $\cL_{\cA(K)}$.  There is a finite set $\chi_\alpha(x)$
of quantifier-free $\cL_{\cA(K)}$--formulas
such that $\Kalg^\circ \models \bigvee_\alpha \chi_\alpha$
(i.e. the sets $\{x: \chi_\alpha(x)\}$ cover $(\Ko_{alg})^n$)
and for each $\alpha$ there is a finite family of
quantifier free formulas $\phi_{\alpha,i}(x,y)$ that depend on $y$ in a polynomial way (i.e. are boolean combinations of
atomic formulas of the form $p(x,y) = 0$ and $|p(x,y)| \leq |q(x,y)|$, where the $p(x,y), q(x,y)$ are polynomials in $y$ with coefficients $\cL_{\cA(K)}$--terms in $x$)
such that, writing $\cU_{\alpha,i} = \cU_{\alpha,i}(x)$ for the set $\{y: \phi_{\alpha,i}(x,y)\}$,
for each $\alpha$ the $\cU_{\alpha,i}(x)$ cover $\Ko_{alg}$,
and for each $\alpha$ and $i$  there are rational functions $R_{\alpha,i}(y)$  with coefficients  $\cL_{\cA(K)}$--terms in $x$, and terms $E_{\alpha,i}(x,y)$ that are strong units on $\cU_{\alpha,i}$ such that for $x$ satisfying $\chi_\alpha$ we have
\begin{eqnarray}\label{A1}
\tau(x,y)\big|_{\cU_{\alpha,i}} = E_{\alpha,i}\cdot R_{\alpha,i}\big|_{\cU_{\alpha,i}}.
\end{eqnarray}
[See \cite{CL1} Definition 5.1.4 for the definition of ``strong unit''.  Here we slightly generalize that definition.  The size of a strong unit $E_{\alpha,i}(x,y)$ is determined by a term $c_{\alpha,i}(x)$ satisfying $|E_{\alpha,i}(x,y)^\ell | = |c_{\alpha,i}(x)|$ for all $y \in \cU_{\alpha,i}$ and there is a polynomial $P_{\alpha,i}$ with coefficients  $\cL_{\cA(K)}$--terms in $x$ such that $||P_{\alpha,i}|| = 1$ and $P_{\alpha,i}(c_{\alpha,i}^{-1}(E_{\alpha,i})^\ell)^{\widetilde{}} = 0$ for all $y \in \cU_{\alpha,i}(x)$. So $c_{\alpha,i}$ and $P_{\alpha,i}$ determine the finitely many possible angular components of $E_{\alpha,i}$.]
 \end{prop}

\begin{proof}
Let $K^*$ be a $|\cL_{\cA(K)}|^+$-- saturated $\cL_{\cA(K)}$--elementary extension of $\Kalg$.  Let $\cD(x)$ be the algebra of $\cL_{\cA(K)}$ terms in the variables $x$.  Then for each $\ox \in (K^*)^n$,  we have that $K \subset \cD(\ox)$ is a field with analytic $\cA(K)$--structure, so \cite{CL1} Theorem 5.5.3 applies and there is a finite set of annuli $\cU_{\ox,i}$ defined by annulus formulas $\phi_{\ox, i}$  (\cite{CL1} Definition 5.1.1) and rational functions $R_{\ox,i}(y)$, strong units $E_{\ox,i}(y)$ and finite exceptional sets $S_{\ox,i}$ satisfying
\begin{eqnarray}\label{A2}
\tau(\ox,y)\big|_{\cU_{\ox,i} \setminus S_{\ox,i}} = E_{\ox,i}\cdot R_{\ox,i}\big|_{\cU_{\ox,i} \setminus  S_{\ox,i}}.
\end{eqnarray}
Observe that the exceptional sets arise from common zeros of numerators and denominators (of subterms in $\tau$) and at these points the term becomes simpler, so the same induction on terms establishes the analogous representation on the exceptional sets, which are defined by polynomial equations in $y$:
\begin{eqnarray}\label{A3}
\tau(\ox,y)\big|_{ S_{\ox,i,j}} = E'_{\ox,i,j}\cdot R'_{\ox,i,j}\big|_{ S_{\ox,i,j}},
\end{eqnarray}
where $S_{\ox,i} = \bigcup_jS_{\ox,i,j}$.
Hence, in the formalism of the Proposition there is no need to consider exceptional sets separately, and we may consider that we have quantifier free formulas $\phi_i(\ox,y)$ depending on $y$ in a polynomial way, and terms $E_{\ox,i}$, that are strong units, and rational functions $R_{\ox,i}$ such that, for all i and all $y \in \cU_{\ox,i}$, the set defined by $\phi_i(\ox,y)$,
\begin{eqnarray}\label{A4}
\tau(\ox,y)\big|_{\cU_{\ox,i}} = E_{\ox,i}\cdot R_{\ox,i}\big|_{\cU_{\ox,i} }.
\end{eqnarray}

Let $E^*_{\ox,i}(x,y)$, $R^*_{\ox,i}(x,y)$ be terms in $\cD(x,y)$ such that $E^*_{\ox,i}(\ox,y) = E_{\ox,i}(y)$
and $R^*_{\ox,i}(\ox,y) = R_{\ox,i}(y)$ and similarly for other terms such as the $c_{\ox,i}$ and $P_{\ox,i}$ and the terms occurring in $\phi_{\ox,i}$ implicit in the  description of \ref{A4}. Denote the resulting formula
corresponding to $\phi_{\ox,i}$ by $\phi^*_{\ox,i}(x,y)$.
Let $\Delta_{\ox}(x)$ be the $\cL_{\cA(K)}$--formula that asserts that for all $i$
\begin{eqnarray}\label{A5}
\tau(x,y)\big|_{\cU_{x,i}} = E^*_{\ox,i}(x,y)\cdot R^*_{\ox.i}(x,y)\big|_{\cU_{x,i} }
\end{eqnarray}
holds, together with  the implied background description about
$E^*_{\ox,i}$ being strong units and $\cU_{x,i}$ being the sets defined by the formulas $\phi^*_{\ox,i}(x,y)$.
  Let $\chi_{\ox}(x)$ be a quantifier-free $\cL_{\cA(K)}$--formula equivalent to $\Delta_{\ox}(x)$ in
$K^{*}_{alg}$. We know that $\chi_{\ox}(x)$ exists by {\it algebraically closed} quantifier elimination (see Remark \ref{AQE}).
Since $K^* \models \chi_{\ox}(\ox)$,
by the saturation of $K^*$ the formulas $\chi_{\ox}(x)$ cover the Stone space of $\cL_{\cA(K)}- n$-types, and by the compactness of that Stone space, finitely many of the $\chi_{\ox}(x)$ cover, say the $\chi_{\alpha}(x)$ for $\alpha$ in some finite subset of $\Kalg^{*\circ}$.
The Proposition now follows.
\end{proof}

In the case of  {\it algebraically closed} fields with separated analytic structure, one can say much more than Proposition \ref{thens} above about definable functions.

\begin{prop}\label{termsalg}
(i) (characteristic $0$)  Let $\cA$ be a separated Weierstrass system and let $x = (x_1,\dots,x_n)$ be several variables and $y$ one variable.  Suppose that a formula $\varphi(x,y) \in \cL_\cA$ defines $y$ as a function of $x$ in all algebraically closed fields of characteristic $0$ with analytic $\cA$--structure.  Then there is a term $\tau$ of $\cL_\cA$ such that
 $L\models\varphi(x,y) \leftrightarrow y = \tau(x)$ for all algebraically closed valued fields $L$ of characteristic $0$ with analytic $\cA$--structure.
\\
(ii) (characteristic $p \neq 0$)  In characteristic $p$ let $\cL'_\cA$ be the language $\cL_\cA$ with the $p$-th root function $(\cdot)^\frac{1}{p}$ adjoined.  The analogous result to (i) holds, with $\cL^{'}_\cA$ in place of $\cL_\cA$, for all algebraically closed valued fields $L$ of characteristic $p$ with analytic $\cA$--structure.
\end{prop}

\begin{proof}
(i).  We may assume that $\varphi$ is quantifier free.  Suppose, for sake of contradiction, that there is no such term $\tau$.  Then, since we can combine terms defined by cases into one term, for any finite set of terms, $\tau_i(x), i=1,\cdots,k$ there is an $x$ such that $\models  \bigwedge_{i=1}^k \neg \varphi(x,\tau_i(x))$.  Hence there is an algebraically closed field $L$ with analytic $\cA$--structure and $x \in L^k$ such that $\models_L \neg  \varphi(x,\tau(x))$ for all terms $\tau(x)$.  Let $\Delta(x) := \{\tau(x): \tau \text{ an } \cL_\cA-\text{term} \} \subset L$.  Let $L'$ be the field of fractions of $\Delta(x)$.  Then $L' \subset L$ is a field with analytic $\cA$--structure and this structure extends uniquely to an analytic $\cA$--structure on $L'_{alg}$.  Hence there is a unique $y \in L'_{alg} \setminus L$ such that $\models_{L'} \varphi(x,y)$.  But, since $L'$ is perfect, there is an analytic $\cA$--structure automorphism of $L'$ that leaves the elements of $L$ fixed but moves $y$, contradicting the uniqueness of $y$.
(Indeed, since $\cA$ is separated $L'$ is henselian, and by \cite{CL1} Theorem 4.5.11 every field automorphism of $L'$ over $L$ is an automorphism of fields with analytic $\cA$--structure).\\
(ii).  Observe that every formula of $\cL'_\cA$ is equivalent to a $\cL_\cA$--formula, and if we define $\Delta'(x) := \{\tau(x): \tau \text{ a } \cL'_\cA \text{ term} \}$, and $L'$ as the field of fractions of $\Delta'(x)$, then $L'$ is a perfect subfield of $L$ so the same argument works.
\end{proof}

\begin{rem}\label{correction2}
We amend Lemmas 6.3.12 and 6.3.14 of \cite{CL1}. In fact, with notation of Lemma 6.3.12 of \cite{CL1}, its proof shows an even stronger conclusion holds, namely that $rv_n(E^\sigma(x))$ only depends on $x \bmod (n \cdot \Kalgoo)$ when $x$ varies over $\Kalgo$. It is in fact this stronger form of Lemma 6.3.12 that is used to prove Lemma 6.3.14 of \cite{CL1}. We thank S.~Rideau for pointing this out to us.
\end{rem}

\begin{rem}\label{pWrem}
In \cite{CL1} Section 4.3 we gave a two-case definition of {\it strictly convergent Weierstrass system $\cA$} and {\it (henselian) field with analytic $\cA$--structure.}  The case (i) (denoted $(\pi = 1)$) of that definition clearly fits into the framework of this paper as an analytic $\cA$--structure with side conditions (namely that $\pi$ be a prime element of $\Ko$) -- see Definitions \ref{side1} and \ref{side2} below. In case (ii) (denoted $(\pi \neq 1)$ in \cite{CL1}, Section 3.3) we considered all rings of fractions $C$ coming from a {\it separated} Weierstrass system containing $\cA$.
In general, some of the elements of such a $C$ will not be definable in the language $\cL_{\cA}$.  The treatment in this paper is more natural in that all the elements and functions that arise are existentially definable in $\cL_\cA$ (see Lemma \ref{hf2bis}), and the definition depends only on the pre-Weierstrass system  $\cA = \{A_m\}$, not on some (unspecified) separated Weierstrass system containing $\cA$.
\end{rem}

\subsection{A discussion of Strong Noetherian Properties}\label{moreSNP}
In this subsection we discuss variants of Definitions \ref{con7} and \ref{SNPstr}

The following conditions (7)$'$ and (vii)$'$ easily imply conditions (7) of Definition \ref{con7} and  (vii) of Definition \ref{SNPstr} and are easier to use in Theorem \ref{ex3}.

\begin{defn}\label{SNPstr'}
 Let $\cB = \{B_{m,n}\}$ be a pre-Weierstrass system, and let $m \leq m', n \leq n'$ and let $\xi = (\xi_1, \cdots , \xi_m)$, $\rho = (\rho_1, \cdots , \rho_n)$, $\xi' = (\xi_1, \cdots , \xi_{m'})$, $\rho' = (\rho_1, \cdots , \rho_{n'})$, $\xi'' = (\xi_{m+1}, \cdots , \xi_{m'})$ and $\rho'' = (\rho_{n+1}, \cdots , \rho_{n'}).$
 \item{(7)$'$} Let
 $$
f = \sum_{\mu,\nu}\overline{f}_{\mu\nu}(\xi,\rho)(\xi'')^\mu
(\rho'')^\nu \in B_{m',n'}
$$
where the $\overline{f}_{\mu\nu}(\xi,\rho) \in B_{m,n}$. There is a
 finite set $J \subset \bN^{m'-m+n'-n}$
and,  for each $(\mu,\nu) \in J$ a function $g_{\mu\nu} \in
(B_{m',n'})^\circ$ such that
$$
f = \sum_{(\mu,\nu) \in J} \overline f_{\mu\nu}(\xi,\rho)
(\xi'')^\mu(\rho'')^\nu (1+g_{\mu\nu})
$$
as an element of $B_{m',n'}$.

\item{(vii)$'$} Let
$$
f = \sum_{\mu,\nu}\overline{f}_{\mu\nu}(\xi,\rho)(\xi'')^\mu
(\rho'')^\nu \in A_{m',n'}^H
$$
where the $\overline{f}_{\mu\nu}(\xi,\rho) \in A_{m,n}^H$. There is a
finite set $J \subset \bN^{m'-m+n'-n}$ and
and,  for each $(\mu,\nu) \in J$ a function $g_{\mu\nu} \in
(A^H_{m',n'})^\circ$ such that
$$
f = \sum_{(\mu,\nu) \in J} \overline f_{\mu\nu}(\xi,\rho)
(\xi'')^\mu(\rho'')^\nu (1+g_{\mu\nu})
$$
as an element of $A^H_{m',n'}$.

\end{defn}

The following is immediate.
\begin{lem}
\item{(i)} If the separated pre-Weierstrass system $\cB$ satisfies condition (7)$'$ above then $\cB$ is a separated Weierstrass system.
\item{(ii)} If the good, strictly convergent pre-Weierstrass system $\cA$ satisfies condition (vii)$'$ above then $\cA$ is a strictly convergent Weierstrass system.
\end{lem}

A weaker Strong Noetherian Property that that is still sufficient for results on quantifier elimination, $b$--minimality and cell decomposition is (7)* below.  See \cite{CL1} section 4.2 for definitions. This definition is Theorem 4.2.15 of \cite{CL1}.
\begin{defn}
\item{(7)*}
 Let $\cB = \{B_{m,n}\}$ be a pre-Weierstrass system, and let $m \leq m', n \leq n'$ and let $\xi = (\xi_1, \cdots , \xi_m)$, $\rho = (\rho_1, \cdots , \rho_n)$, $\xi' = (\xi_1, \cdots , \xi_{m'})$, $\rho' = (\rho_1, \cdots , \rho_{n'})$, $\xi'' = (\xi_{m+1}, \cdots , \xi_{m'})$ and $\rho'' = (\rho_{n+1}, \cdots , \rho_{n'}),$ and let
$$
f = \sum_{\mu,\nu}\overline{f}_{\mu\nu}(\xi,\rho)(\xi'')^\mu
(\rho'')^\nu \in B_{m',n'}
$$
where the $\overline{f}_{\mu\nu}(\xi,\rho) \in B_{m,n}$. (The $\overline{f}_{\mu\nu}(\xi,\rho)$ are well defined by Weierstrass division in  $B_{m',n'}$). There is a
finite system $\cF$ of rings of $B$-fractions, and for each $B' \in
\cF$ there is a finite, disjointly covering family of Laurent rings
$C$ over $A'_{m,n}$, and, for each $C$, there are: a finite set $J_C$ and
$C$-units $u_{C\mu\nu} \in C$  and functions $h_{C\mu\nu} \in
C_{m'-m,n'-n}^\circ$ for $(\mu,\nu) \in J_C$ such that
$$
f = \sum_{(\mu,\nu) \in J_C} \overline f_{\mu\nu}(\xi,\rho)
(\xi'')^\mu(\rho'')^\nu u_{C\mu\nu}(1+h_{C\mu\nu})
$$
as an element of $C_{m'-m,n'-n}$.

\item{(vii)*} The corresponding condition for a strictly convergent pre-Weierstrass system $\cA$ is that the separated pre-Weierstrass system $\cA^H$ satisfy condition (7)*.

\end{defn}

\begin{defn}\label{easier}The arguments in  \cite{CL1} section 4.2 (Lemma 4.2.14 and Theorem 4.2.15) show that the following less cumbersome condition implies condition (7)*:
\item[(7)$''$] For any
$$
f = \sum_{\mu,\nu}\overline{c}_{\mu\nu}(\xi)^\mu
(\rho)^\nu \in B_{m,n}
$$
 there is a
system $\cF$ of rings of $\cB$-fractions, and for each $B' \in \cF$
 there is a finite set $J_{B'}$,
and functions $g_{\mu\nu} \in
(B'_{m,n})^\circ$ for $(\mu,\nu) \in J_{B'}$ such that
$$
f = \sum_{(\mu,\nu) \in J_{B'}} \overline c_{\mu\nu}\xi^\mu\rho^\nu
(1+g_{\mu\nu})
$$
as an element of $B'_{m,n}$.
\end{defn}

We do not know an example of a good pre-Weierstrass system $\cA$ such that $\cA^H := \{A_{m,n}^H\}$ is a separated Weierstrass System but does not satisfy condition(vii)$'$ of Definition \ref{SNPstr'}.  Indeed, we do not have examples that distinguish among the various strong noetherian properties below.

\subsection{Extension of parameters}\label{ExtPar}

Let $\cB$ be a separated pre-Weierstrass system and let $K$ be a field with analytic $\cB$--structure via $\sigma$. As in \cite{CL1} Definition 4.5.6(i), assume that $\sigma_0 : A \to \Ko$ is an embedding, and define
$$
B_{m,n}(K) := \{f^\sigma(\xi,\rho, c, d) : c \in(\Ko)^{M}, d \in (\Koo)^{N} \text{ and } f \in B_{m+M,n+N} \}.
$$
Then $B_{m,n}(K)$ is a ring of functions $(\Ko)^m \times (\Koo)^n \to K$.  If $\cB$ satisfies condition (7)* above (for example if $\cB$ is a Weierstrass system) it follows that the homomorphism $f^\sigma \mapsto \cS(f^\sigma)$ is an isomorphism, so we may regard $B_{m,n}(K)$ as a ring of power series over $\Ko$.  Then $\cB(K) := \{B_{m,n}(K)\}$ is a Weierstrass system by \cite{CL1} Theorem 4.5.7(i).

In Definition 4.5.6(ii) of \cite{CL1}, a similar extension of parameters is given for some strictly convergent analytic pre-Weierstrass systems, i.e. those for which both $I$ and $\Koo$ have a prime element. As a further example, complementing the ones in \ref{ex}, we have the following extension of
\cite{CL1} Theorem 4.5.7(ii), where we no longer impose that condition on $I$ or $K$.
\begin{lem}
Let $K$ be a henselian field with analytic $\cA$--structure for strictly convergent Weierstrass system $\cA$.
Then also the system $\cA(K)$ obtained by extension of parameters
is a strictly convergent Weierstrass System.
\end{lem}
\begin{proof}
Because $\Ko$ is a valuation ring no rings of fractions are needed and we have that $\cA(K)^H = \cA^H(K)$.
\end{proof}

We will give another strong noetherian property ((7)$'''$ below) that has the advantage that it is expressed in terms of the fields with analytic $\cB$--structure, rather than the somewhat awkward concept of systems of rings of fractions.  It is useful in defining the concept of Weierstrass system with side conditions in subsection \ref{side} below.  Let $\cB$ be a pre-Weierstrass system.
For any (henselian) field $K$ with analytic $\cB$--structure, let
$$
K_{\cB}^\circ := \{\tau : \tau \text{  a variable-free term of  } \cL_\cB^D \}
$$
and let $K_\cB$ be the field of fractions of $K_\cB^\circ$.

\begin{defn}
Let $\cB$ be a pre-Weierstrass system and let  $\cB(K_\cB)$ be obtained by extending parameters.  \item[(7)$'''$]
If $f^\sigma  \in B_{m,n}(K_\cB)$ and
$\cS(f^\sigma) = \sum_{\mu,\nu}\overline{c}_{\mu\nu}\xi^\mu\rho^\nu$ with the $\overline{c}_{\mu\nu} \in \Ko_\cB$, then there is a finite set
$J \subset \bN^{m+n}$ and for each $(\mu,\nu) \in J$ there is a
 $g^\sigma_{\mu\nu} \in (B_{m,n}(K_\cB))^\circ$ such that
\begin{eqnarray}\label{SNPeqn}
f^\sigma = \sum_{(\mu,\nu) \in J} \overline{c}_{\mu\nu}\xi^\mu\rho^\nu(1+g^\sigma_{\mu\nu}).
\end{eqnarray}
\end{defn}

\begin{rem}
\item{a)} When $B = \Ko$ for $K$ a (henselian) field all the rings of fractions are in fact just $\Ko$ so in that case the definitions become much simpler.
\item{b)} Condition (7)$'''$ gives us for each $f$ a (possibly different) strong noetherian property (equation \ref{SNPeqn}) in each field $K$ with analytic $\cB$--structure.  A standard compactness argument shows that there is a finite set of representations of the form \ref{SNPeqn} such that in every field $K$ with analytic $\cB$--structure (or in the case, below, of Weierstrass systems with side conditions, in every such $K$ satisfying additional axioms $\cT$) one of the finite set of equations (\ref{SNPeqn}) holds in $K$.
\end{rem}

\subsection{Weierstrass systems with side conditions}\label{side}
The following small extension of the above concepts may be useful.  See the discussion preceeding Theorem \ref{ex3}.

\begin{defn}A (strictly convergent, or separated) {\it pre-Weierstrass system with side conditions} is a (good strictly convergent, or separated) pre-Weierstrass system $\cA$, together with some axioms $\cT_\cA$ in the language $\cL_\cA$.  \label{side1}
\end{defn}

\begin{defn}A {\it separated Weierstrass system with side conditions} is a separated pre-Weierstrass system $\cB$, together with axioms  $\cT_\cB$ in the language $\cL_\cA$, such that, for each henselian field $K$ with analytic $\cB$--structure that satisfies the axioms of $\cT_\cB$, the pre-Weierstrass system $\cB(K_\cB)$ obtained by extension of parameters (cf. Subsection \ref{ExtPar} above) satisfies the condition
\item[(7)$''$] If $f = \sum_{\mu,\nu}\overline{c}_{\mu\nu}\xi^\mu\rho^\nu \in B_{m,n}(K_\cB)$
 with the $\overline{c}_{\mu\nu} \in \Ko_\cB$, then there is a finite set
$J \subset \bN^{m+n}$ and for each $(\mu,\nu) \in J$ there is a
 $g_{\mu\nu} \in (B_{m,n}(K_\cB))^\circ$ such that
$$
f = \sum_{(\mu,\nu) \in J} \overline{c}_{\mu\nu}\xi^\mu\rho^\nu(1+g_{\mu\nu}).
$$
In this situation we refer to the analytic structure on $K$ as an {\it analytic structure with side conditions.}  \label{side2}
We call a good strictly convergent pre-Weierstrass system $\cA$ a {\it strictly convergent Weierstrass system with side conditions} if the seperated pre-Weierstrass system $\cA^H$ satisfies the above condition for evey henselian field $K$ with analytic $\cA$--structure that satisfies the side conditions.
\end{defn}

The following is immediate.

\begin{cor}Let $\cA$ be a strictly convergent Weierstrass system with side conditions $\cT_\cA$.  There is a set of existential definitions in the the langage $\cL_\cA^D$ which, in every henselian field $K$ with analytic $\cA$--structure with side conditions $\cT_\cA$, define a separated analytic structure on $K$ extending the  analytic $\cA$--structure.  Indeed, the functions are given by terms of $\cL_\cA^{D,h}$.
\end{cor}

\subsection*{Acknowledgments}
\hspace{0.5cm}
The authors would like to thank Silvain Rideau for interesting discussions about the amendments to \cite{CL1} in Appendix A.1. We would also like to thank F. Martin for discussions that led to Proposition \ref{termsalg} and Theorem \ref{terms}.
The authors were supported in part by the European Research Council under the European Community's Seventh Framework Programme (FP7/2007-2013)
with ERC Grant Agreements nrs. 246903 NMNAG and 615722 MOTMELSUM,
by the Labex CEMPI  (ANR-11-LABX-0007-01), and by the the Fund for Scientific Research of Flanders, Belgium (grant G.0415.10). The authors would also like to thank the IHES and the FIM of the ETH in Z\"urich, where part of the research was done.

\end{document}